\documentclass[10pt, a4paper, twoside]{amsart}

\date{\today}

\usepackage[usenames, dvipsnames]{color}
\definecolor{darkblue}{rgb}{0.0, 0.0, 0.45}

\usepackage[colorlinks	= true,
			raiselinks	= true,
			linkcolor	= darkblue, 
			citecolor	= Mahogany,
			urlcolor	= ForestGreen,
			pdfauthor	= {Peyman Mohajerin Esfahani},
			pdftitle	= {PhD Thesis},
			pdfkeywords	= {},
			pdfsubject	= {PhD Thesis},
			plainpages	= false]{hyperref}

\usepackage{dsfont,amssymb,amsmath,subfigure, graphicx} 

\usepackage{enumitem}

\usepackage{amsfonts,dsfont,mathtools, mathrsfs,amsthm} 
\usepackage[amssymb, thickqspace]{SIunits}
\usepackage{algorithm,algorithmicx}

\allowdisplaybreaks
\renewcommand\thesection{\arabic{section}}

\newtheorem{Thm}{Theorem} [section]
\newtheorem{Prop}[Thm]{Proposition}
\newtheorem{Fact}[Thm]{Fact}
\newtheorem{Lem}[Thm]{Lemma}

\newtheorem{As}[Thm]{Assumption}

\newtheorem{Def}[Thm]{Definition}

\newtheorem{Rem}[Thm]{Remark}

\newcommand{\PP}{\mathds{P}}
\newcommand{\EE}{\mathds{E}}
\newcommand{\R}{\mathbb{R}}
\newcommand{\N}{\mathbb{N}}
\newcommand{\eps}{\varepsilon}

\newcommand{\lra}{\longrightarrow}

\newcommand{\ra}{\rightarrow}
\newcommand{\da}{\downarrow}
\newcommand{\ua}{\uparrow}

\newcommand{\ind}[1]{\mathds{1}_{#1}}
\newcommand{\Let}{\coloneqq}
\newcommand{\teL}{\eqqcolon}
\newcommand{\diff}{\mathrm{d}}
\newcommand{\mn}{\wedge}
\newcommand{\mx}{\vee}
\newcommand{\ol}[1]{\overline{#1}}
\newcommand{\ul}[1]{\underline{#1}}
\newcommand{\wt}{\widetilde}
\newcommand{\shift}{\vartheta}
\newcommand{\com}{\circ}
\newcommand{\tr}{^\intercal}

\newcommand{\eqsmall}[1]{{\small $#1$}}
\newcommand{\goodgap}{\hspace{\subfigtopskip}}

\newcommand{\traj}[4]{#1_{#4}^{#2;#3}}
\newcommand{\setofst}[1]{\mathcal{T}_{[#1]}}
\newcommand{\sigalg}{\mathcal{F}}
\newcommand{\filtration}{\mathds{F}}

\newcommand{\sat}{\models}

\newcommand{\borel}{\mathfrak{B}}
\newcommand{\Reach}[1]{\stackrel{#1}{\lra}}
\newcommand{\Path}{\leadsto}
\newcommand{\PathSet}{\mathrm{MP}}

\newcommand{\controlmaps}[1]{\mathcal{#1}}
\newcommand{\control}[1]{\boldsymbol{#1}}
\newcommand{\ball}[1]{\mathrm{#1}}

\newcommand{\set}[1]{\mathbb{#1}}

\newcommand{\E}[1]{\mathds{E}{\left[ #1 \right]} }
\DeclareMathOperator{\dist}{dist}

\usepackage{fancyhdr}
\title{Motion Planning for Continuous Time Stochastic Processes: A Dynamic Programming Approach}
\thanks{PME and JL are with the Automatic Control Laboratory, ETH Z\"urich, 8092 Z\"urich, Switzerland; DC is with the Systems \& Control Engineering, IIT-Bombay, Powai, Mumbai 400076, India. Emails: {\{mohajerin,lygeros\}@control.ee.ethz.ch, chatterjee@sc.iitb.ac.in}}

\author{Peyman Mohajerin Esfahani, Debasish Chatterjee, and John Lygeros}

\begin{document} 
	\maketitle
	\begin{abstract}
		We study stochastic motion planning problems which involve a controlled process, with possibly discontinuous sample paths, visiting certain subsets of the state-space while avoiding others in a sequential fashion. For this purpose, we first introduce two basic notions of motion planning, and then establish a connection to a class of stochastic optimal control problems concerned with sequential stopping times. A weak dynamic programming principle (DPP) is then proposed, which characterizes the set of initial states that admit a control enabling the process to execute the desired maneuver with probability no less than some pre-specified value. The proposed DPP comprises auxiliary value functions defined in terms of discontinuous payoff functions. A concrete instance of the use of this novel DPP in the case of diffusion processes is also presented. In this case, we establish that the aforementioned set of initial states can be characterized as the level set of a discontinuous viscosity solution to a sequence of partial differential equations, for which the first one has a known boundary condition, while the boundary conditions of the subsequent ones are determined by the solutions to the preceding steps. Finally, the generality and flexibility of the theoretical results are illustrated on an example involving biological switches.
	\end{abstract}

\section{Introduction}

	Motion planning of dynamical systems aims to steer the state of the system through certain given sets in a specific order and pre-assigned time schedule. This problem finds a wide spectrum of applications ranging from air traffic  management \cite{ref:TomlinLygerosSastry-2000} and security of power networks \cite{ref:Mohajerin_ACC10} to navigation of unmanned air vehicles \cite{ref:Beard-00,ref:Beard-02}. 
	
	\subsection{Context}	
	The two fields of robotics and control have contributed much to this problem. Particular problems of interest in this context are the control synthesis for a given initial condition to accomplish the motion planning task, and the determination of the set of initial conditions for which such a controller exists.

	The focus in the robotics community has mainly been on the first objective (i.e., control synthesis for a given initial condition) with particular emphasis on computational issues. To this end, the motion planning problem is typically \emph{approximated} in an optimal control framework wherein the cost function rewards the target set while penalizing the obstacles. In this vein, there is a rich literature to tackle the corresponding approximation problems. Examples include \cite{ref:DPP-Mayne-70} that capitalizes on the synergy between local and global methods, \cite{ref:Todorov} that builds on \cite{ref:DPP-Mayne-70} to propose an iterative LQG scheme, and \cite{ref:Kappen-07} extended later by \cite{ref:BuchSchaal-10} that proposes a reinforcement learning based approach to generate parametrized control policy via path integrals. 
	
	Research toward the second objective (i.e. characterizing the desired set of initial conditions) is traditionally conducted in the control community, and is the focus of this article. A technical difficulty to \emph{accurately} characterize this set is a potential non-smooth behavior commonly arising in the optimal control context, which consequently renders the analysis of the value function more involved. This issue is a computational concern for control synthesis purposes as well; see, for example, \cite{ref:Milut-11} that opts to circumvent this non-smoothness by introducing additional noise. 
	
	\subsection{Literature on Set Characterization}
	From a set characterization viewpoint, motion planning problems in the deterministic setting have been studied extensively from different perspectives; see, for example, \cite{ref:Aubin-1991} in the language of viability theory, \cite{ref:Shin-86} for a dynamic programming approach, and also \cite{ref:Suss-motion-91,ref:Suss-motion-98} for a differential geometric perspective. In the stochastic setting, stochastic viability and controlled invariance are studied in an almost-sure treatment in \cite{ref:AubinPratoFrankowska-2000, ref:BardiJensen-2002}; see also the references therein. Following the same objective, methods involving stochastic contingent sets \cite{ref:AubinPrato-1998}, viscosity solutions to second-order partial differential equations \cite{ref:BardiGoatin-1999}, and also the equivalence for the invariance problem between a stochastic differential equation (SDE) and a certain deterministic control system were developed in this context \cite{ref:PratoFrankowska-2004}. 
	
	Although almost-sure motion planning maneuvers are interesting in their own right, they maybe be too conservative in applications with less stringent safety requirements where a common specification involves only bounding the probability that undesirable events take place. This line of research in the stochastic setting has received relatively little attention, in particular for systems governed by SDEs. In fact, it was not until recently that the basic motion planning problem involving one target and one obstacle set, the so-called reach-avoid problem, was investigated in the context of finite probability spaces for a class of continuous-time Markov decision processes \cite{ref:Katoen-05}, and in the discrete-time stochastic hybrid systems context \cite{Chatterjee2011, ref:Summers-10}.

	In the continuous time and continuous space settings, one may tackle the dynamic programming formulation of the reach-avoid problem from two perspectives: a direct technique based on the theory of stochastic target problems, and an indirect approach via an exit-time stochastic optimal control formulation. For the former, we refer the reader to \cite{BouchardTouzi_WeakDPP}; see also the recent book \cite{ref:Touzi-13}. In our earlier works \cite{ref:MohChaLyg-11,ref:MohChatLyg-15} we focused on the latter perspective for reachability of controlled diffusion processes. Here we continue in the same spirit by going beyond the reach-avoid problem to more complex motion planning problems for a larger class of stochastic processes with possibly discontinuous sample paths. Preliminary results in this direction were reported in \cite{ref:MohMilCha-2013} with an emphasis on the application side and without covering the technical details.

	\subsection{Our Contributions}
	
	In this context the main contributions of the present article are summarized below:	
	
	\begin{enumerate}[label=(\roman*), itemsep = 2mm, nolistsep, leftmargin=*, align=right, widest=iii]
		
		\item Based on a formal definition of the different stochastic motion planning maneuvers (Section \ref{sec:problem}), we establish a connection between the desired task and a class of stochastic optimal control problems (Section \ref{sec:connection}); 
		
		\item we propose a weak dynamic programming principle (DPP) under mild assumptions on the admissible controls and the stochastic process (Section \ref{sec:DPP}); 
		
		\item in the context of controlled SDEs, based on the proposed DPP we derive a partial differential equation (PDE) characterization of the desired set of initial conditions (Section \ref{sec:application}). 
	\end{enumerate}

	In more detail, we start with the formal definition of a motion planning objective comprising of two fundamental reachability maneuvers. To the best of our knowledge, this is new in the literature. We address the following natural question: for which initial states do there exist an admissible control such that the stochastic processes satisfy the motion planning objective with a probability greater than a given value $p$? We then characterize this set of initial states by establishing a connection between the motion planning specifications and a class of stochastic optimal control problems involving discontinuous payoff functions and a sequence of successive exit-times. 
	
	Due to the discontinuity of the payoff functions, the classical results in stochastic optimal control problems and its connection to Hamilton-Jacobi-Bellman PDEs are not applicable here, see for instance \cite[Section IV.7]{SonerBook}. Under mild technical assumptions, we propose a weak DPP involving auxiliary value functions which only requires semicontinuity of the value functions. It is worth noting that the DPP is developed in a rather general setting, emphasizing particular features of the underlying process that lead to such a characterization of the motion planning. Closest in spirit to our work is \cite{BouchardTouzi_WeakDPP} in which the optimal control problem is studied in a fixed time horizon as well as an optimal stopping setting. Neither of these settings is applicable to our exit-time framework. From a technical standpoint, this article extends the technology developed by \cite{BouchardTouzi_WeakDPP} to deal with the class of exit-time problems suitable to our motion planning tasks. 
		

	
	
	Finally, we focus on a class of controlled SDEs in which the required assumptions of the proposed DPP are investigated. Indeed, it turns out that the standard uniform non-degeneracy and exterior cone conditions of the involved sets suffice to fulfill the DPP requirements. Subsequently, we demonstrate how the DPP leads to a new framework for characterizing the desired set of initial conditions based on tools from PDEs. Due to the discontinuities of the value functions involved, all the PDEs are understood in the generalized notion of the so-called discontinuous viscosity solutions. In this context, we show how the value functions can be solved by a means of a sequence of PDEs, in which the preceding PDE provides the boundary condition of the current one. 
	
	\subsection{Computational Complexity and Existing Numerical Tools}
	On the computational side, it is well-known that PDE techniques suffer from the curse of dimensionality. In the literature a class of suboptimal control methods referred to as Approximate Dynamic Programming (ADP) have been developed for dealing with this difficulty; for a sampling of recent works see \cite{ref:FarVan-03} for a linear programming approach, \cite{ref:Konda-03} for actor-critic algorithms, and \cite{ref:Bertsekas-05} for a comprehensive survey on the entire area. Besides the ADP literature, very recent progress on numerical methods based on tensor train decompositions holds the potential of substantially ameliorating this curse of dimensionality; see a representative article \cite{ref:KhoSch-11} and the references therein. In this light, taken in its entirety, the results in this article can be viewed as a theoretical bridge between the motion planning objective formalized in Section \ref{sec:problem} and sophisticated numerical methods that can be used to address real problem instances. Here we demonstrate the practical use of this bridge by addressing a stochastic motion planning problem for biological switches.
	

	The organization of the article is as follows: In Section \ref{sec:problem} we formally introduce the stochastic motion planning problems. In Section \ref{sec:connection} we establish a connection between the motion planning objectives and a class of stochastic optimal control problems, for which a weak DPP is proposed in Section \ref{sec:DPP}. A concrete instance of the use of the novel DPP in the case of controlled diffusion processes is presented in Section V, leading to characterization of the motion planning objective with the help of a sequence of PDE's in an iterative fashion. To validate the performance of the proposed methodology, in Section \ref{sec:simulation} the theoretical results are applied to a biological two-gene network. For better readability, the technical proofs along with required preliminaries are provided in the appendices.

\paragraph{Notation}
Given $a,b \in \R$, we define $a\mn b \Let \min\{a, b\}$ and $a\mx b \Let \max\{a, b\}$. We denote by $A^c$ (resp.\ $A^\circ$) the complement (resp.\ interior) of the set $A$. We also denote by $\ol{A}$ (resp.\ $\partial A$) the closure (resp.\ boundary) of $A$. We let $\ball{B}_r(x)$ be an open Euclidean ball centered at $x$ with radius $r$. The Borel $\sigma$-algebra on a topological space $\set{A}$ is denoted by $\borel(\set{A})$, and measurability on $\R^d$ will always refer to Borel-measurability. Given function \(f:\set{A}\ra\R\), the lower and upper semicontinuous envelopes of \(f\) are defined, respectively, by $f_{*}(x) := \liminf_{x' \ra x} f(x')$ and $f^{*}(x) := \limsup_{x' \ra x} f(x')$. Throughout this article all (in)equalities between random variables are understood in almost sure sense. For the ease of the reader, we also provide here a partial notation list which will be also explained in more details later throughout the article:
\begin{enumerate}[label=$\bullet$, itemsep = 0mm, topsep=0mm] 
	\item $\controlmaps{U}_t$: set of admissible controls at time $t$;
	\item $(\traj{X}{t,x}{\control{u}}{s})_{s \ge 0}$: stochastic process under the control $\control{u}$ and convention $\traj{X}{t,x}{\control{u}}{s} \Let x$ for all $s \le t$;
	\item $(W_i \Path G_i)_{\le T_i}$ (resp.\ $W_i \Reach{T_i} G_i$ ) : motion planning event of reaching $G_i$ sometime before time $T_i$ (resp.\ at time $T_i$) while staying in $W_i$, see Definition \ref{def:event};
	\item $\big(\Theta^{A_{k:n}}_i\big)^n_{i = k}$: sequential exit-times from the sets $(A_i)^n_{i = k}$ in order, see Definition \ref{def:theta};
	\item $\mathcal{L}^u$: Dynkin operator, see Definition \ref{def:Dynkin operator}.
\end{enumerate}

\section{General Setting and Problem Description} \label{sec:problem}

	Consider a filtered probability space $(\Omega, \sigalg, \filtration, \PP)$ whose filtration $\filtration \Let (\sigalg_s)_{s\ge 0}$ is generated by an $\R^{d_z}$-valued process $\control{z} \Let (z_s)_{s\ge 0}$  with independent increments. Let this natural filtration be enlarged by its right-continuous completion, i.e., it satisfies the usual conditions of completeness and right continuity \cite[p.\ 48]{ref:KarShr-91}. Consider also an auxiliary subfiltration $\filtration_t \Let (\sigalg_{t,s})_{s\ge0}$, where $\sigalg_{t,s}$ is the $\PP$-completion of $\sigma\big(z_{r \mx t} - z_t, r \in [0,s]\big)$. 
	It is obvious to observe that any $\sigalg_{t,s}$-random variable is independent of $\sigalg_t$, $\sigalg_{t,s} \subseteq \sigalg_s$ with equality in case of $t=0$, and for $s \le t$, $\sigalg_{t,s}$ is the trivial $\sigma-$algebra.

	The object of our study is an $\R^d$-valued controlled random process $\big(\traj{X}{t,x}{\control{u}}{s}\big)_{s\ge t}$, initialized at $(t,x)$ under the control $\control{u} \in \controlmaps{U}_t$, where $\controlmaps{U}_t$ is the set of admissible controls at time $t$. Since the precise class of admissible controls does not play a role until Section IV we defer the formal definition of these until then. Let $T>0$ be a fixed time horizon, and let $\set{S} \Let [0,T] \times \R^d$. We assume that for every $(t,x) \in \set{S}$ and $\control{u} \in \controlmaps{U}_t$, the process $\big(\traj{X}{t,x}{\control{u}}{s}\big)_{s\ge t}$ is $\filtration_t$-adapted process whose sample paths are right continuous with left limits. We denote by $\mathcal{T}$ the collection of all $\filtration$-stopping times; for $\tau_1, \tau_2 \in \mathcal{T}$ with $\tau_1 \le \tau_2$ $\PP$-a.s.\ we let the subset $\setofst{\tau_1,\tau_2}$ denote the collection of all $\filtration_{\tau_1}$-stopping times $\tau$ such that $\tau_1 \leq \tau \leq \tau_2$ $\PP$-a.s. 


	Given sets $(W_i, G_i) \in \borel(\R^d) \times \borel(\R^d)$ for $i \in \{1, \cdots,n\}$, we are interested in a set of initial conditions $(t,x) \in \set{S}$ such that there exists an admissible control $\control{u} \in \controlmaps{U}_t$ steering the process $\traj{X}{t,x}{\control{u}}{\cdot}$ through $(W_i)_{i = 1}^n$ (``way point'' sets) while visiting $(G_i)_{i = 1}^{n}$ (``goal'' sets) in a pre-assigned order. One may pose this objective from different perspectives based on different time scheduling for the excursions between the sets. We formally introduce some of these notions which will be addressed throughout this article. 

	\begin{Def} [Motion Planning Events]
	\label{def:event}
		Consider a fixed initial condition $(t,x) \in \set{S}$ and admissible control $\control{u} \in \controlmaps{U}_t$. Given a sequence of pairs $(W_i,G_i)_{i=1}^n \subset \borel(\R^d)\times \borel(\R^d)$ and horizon times $(T_i)_{i=1}^n \subset [t,T]$, we introduce the following \emph{motion planning events}:
		{\begin{subequations}		
		\label{event}
			\begin{flalign}
				\label{path-event} \Big \{ & \traj{X}{t,x}{\control{u}}{\cdot}  \sat \big[ ( W_1 \Path G_1) \com \cdots \com (W_n \Path G_n) \big]_{\le T} \Big \} \Let 
				\\
				&  \quad \Big\{ \exists(s_i)_{i=1}^{n} \subset [t,T] ~\big |~ \notag \traj{X}{t,x}{\control{u}}{s_i} \in G_i  \quad \text{and}  
				\quad \traj{X}{t,x}{\control{u}}{r} \in W_i \setminus G_i, ~\forall r \in [s_{i-1}, s_i[, \quad \forall i \le n \Big \}, \nonumber 
				\\
				\label{reach-event} \Big \{& \traj{X}{t,x}{\control{u}}{\cdot}  \sat (W_1 \Reach{T_1} G_1) \com \cdots \com (W_n\Reach{T_n} G_n) \Big \} \Let 
				\\
				&  \quad \qquad \Big\{ \traj{X}{t,x}{\control{u}}{T_i} \in G_i \quad \text{and} \quad \traj{X}{t,x}{\control{u}}{r} \in W_i, \notag \quad \forall r \in [T_{i-1}, T_i], \quad \forall i \le n \Big \}, \nonumber	
			\end{flalign}
		\end{subequations}}
		where $s_0 = T_0 \Let t$.		
	\end{Def}

	The set in \eqref{path-event}, roughly speaking, contains the events in the underlying probability space that the trajectory $\traj{X}{t,x}{\control{u}}{\cdot}$, initialized at $(t,x) \in \set{S}$ and controlled via $\control{u} \in \controlmaps{U}_t$, succeeds in visiting $(G_i)_{i = 1}^{n}$ in a certain order, while the entire duration between the two visits to $G_{i-1}$ and $G_{i}$ is spent in $W_{i}$, all within the time horizon $T$. In other words, the journey from $G_{i-1}$ to the next destination $G_i$ must belong to $W_i$ for all $i$. Figure \ref{fig:event:path} depicts a sample path that successfully contributes to the first three phases of the excursion in the sense of \eqref{path-event}. In the case of \eqref{reach-event}, the set of paths is usually more restricted in comparison to \eqref{path-event}. Indeed, not only is the trajectory confined to $W_i$ on the way between $G_{i-1}$ and $G_i$, but also there is a time schedule $(T_i)_{i=1}^n$ that a priori forces the process to be at the goal sets $G_i$ at the specific times $T_i$ for each $i$. Figure \ref{fig:event:reach} demonstrates one sample path in which the first three phases of the excursion are successfully fulfilled.

	\begin{figure}[t!]
	\centering
		\subfigure[A sample path satisfying the first three phases of the specification in the sense of \eqref{path-event}]{\label{fig:event:path}\includegraphics[scale = 0.4]{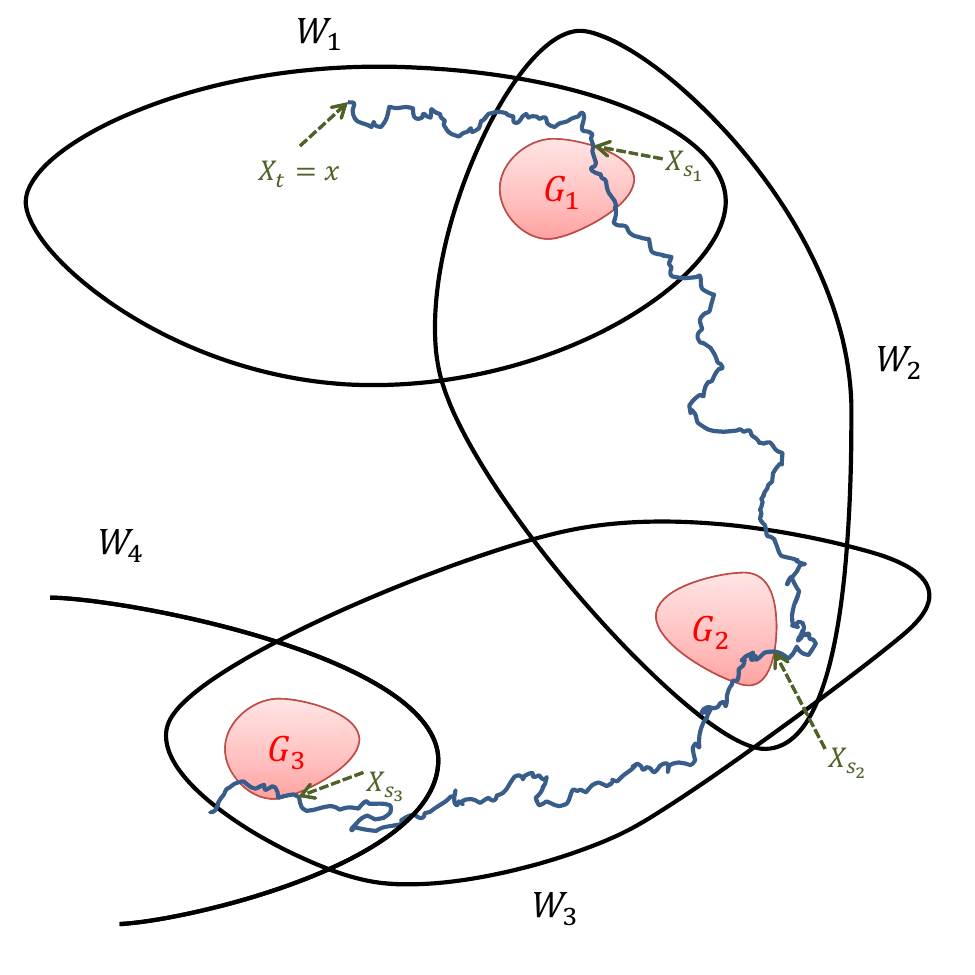}} 
		\quad
		\subfigure[A sample path satisfying the first three phases of the specification in the sense of \eqref{reach-event}]{\label{fig:event:reach}\includegraphics[scale = 0.4]{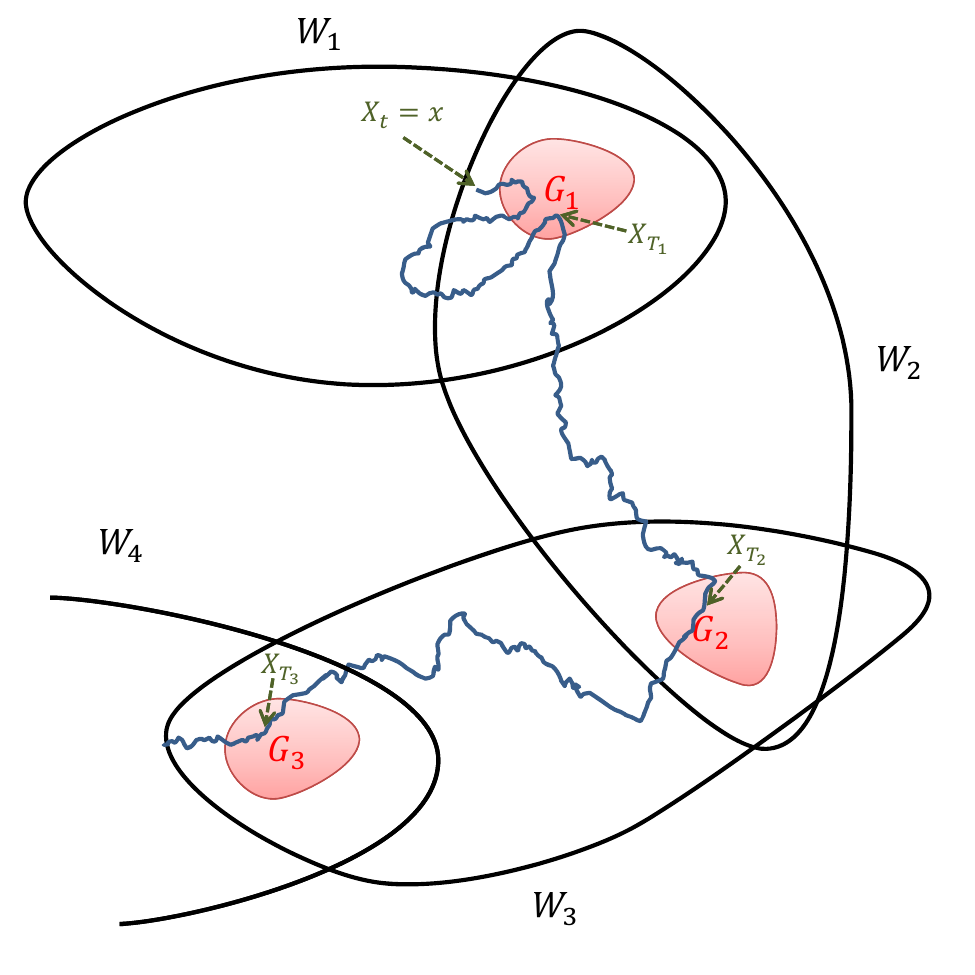}}
		\caption{Sample paths of the process $\traj{X}{t,x}{\control{u}}{\cdot}$ for a fix control $\control{u}\in \controlmaps{U}_t$}
		\label{fig:event}
	\end{figure}

	Note that once a trajectory belonging to the set in \eqref{path-event} visits $G_{i}$ for the first time, it is required to remain in $W_{i+1}$ until the next goal $G_{i+1}$ is reached, whereas a trajectory belonging to the set in definition \eqref{reach-event} may visit the destination $G_i$ several times, while staying in $W_i$ until the intermediate time schedule $T_i$. The only requirement, in contrast to \eqref{path-event}, is to confine the trajectory to be at $G_i$ at the time $T_i$. As an illustration, one can easily inspect that the sample path in Figure \ref{fig:event:reach} indeed violates the requirements of the definition \eqref{path-event} as it leaves $W_2$ after it visits $G_1$ for the first time. In other words, the definition \eqref{path-event} changes the admissible way set $W_i$ to $W_{i+1}$ \emph{immediately after} the trajectory visits $G_i$, while the definition  \eqref{reach-event} only changes the admissible way set \emph{only after} the intermediate time $T_i$ irrespective of whether the trajectory visits $G_i$ prior to $T_i$.

	
	For simplicity we may impose the following assumptions:
		\begin{As}
		\label{a:set}
		We stipulate that 
		\begin{enumerate} [label=\alph*., itemsep = 2mm] 
			\item \label{a:set:closed} The sets $(G_i)_{i=1}^n \subset \borel(\R^d)$ are closed.
			\item \label{a:set:open} The sets $(W_i)_{i=1}^n \subset \borel(\R^d)$ are open.
		\end{enumerate}
		\end{As}

	Concerning Assumption \ref{a:set}.\ref{a:set:closed}, if $G_i$ is not closed, then it is not difficult to see that there could be some continuous transitions through the boundary of $G_i$ that are not admissible in view of the definition \eqref{path-event} since the trajectory must reside in $W_i \setminus G_i$ for the whole interval $[s_{i-1}, s_i[$ and just hit $G_i$ at the time $s_i$. Notice that this is not the case for the definition \eqref{reach-event} since the trajectory only visits the sets $G_i$ at the specific times $T_i$ while any continuous transition and maneuver inside $G_i$ are allowed. Assumption \ref{a:set}.\ref{a:set:open} is rather technical and required for the analysis employed in subsequent sections.
	


	The events introduced in Definition \ref{def:event} depend, of course, on the control $\control{u} \in \controlmaps{U}$ and initial condition $(t,x)\in \set{S}$. The main objective of this article is to determine the set of initial conditions $x \in \R^d $ such that there exists an admissible control $\control{u}$ where the probability of the motion planning events is higher than a certain threshold. Let us formally introduce these sets as follows:

	\begin{Def}[Motion Planning Initial Condition Set]
	\label{def:initial set}
		Consider a fixed initial time $t \in [0,T]$. Given a sequence of set pairs $(W_i,G_i)_{i=1}^n \subset \borel(\R^d) \times \borel(\R^d)$ and horizon times $(T_i)_{i = 1}^n \subset [t,T]$, we define the following \emph{motion planning initial condition sets}:
		\begin{subequations}		
		\label{path-reach-set}
			\begin{flalign}
				\label{path-set}  \PathSet \big(t,p ; (W_i,G_i)_{i=1}^n,  T \big) & \Let 
				\Big \{ x \in \R^d ~ \big | ~ \exists \control{u} \in \controlmaps{U}_t : \\
				& \quad \PP \big\{ \traj{X}{t,x}{\control{u}}{\cdot} \sat \big[( W_1 \Path G_1) \com \cdots \com (W_n \Path G_n) \big]_{\le T} \} >p \Big\}, \nonumber
				\\
				\label{reach-set} \wt \PathSet \big(t,p ; (W_i,G_i)_{i=1}^n,(T_i)_{i=1}^n \big) & \Let \Big \{ x \in \R^d \big | \exists \control{u} \in \controlmaps{U}_t: \\
				& \quad  \PP \big\{ \traj{X}{t,x}{\control{u}}{\cdot} \sat (W_1 \Reach{T_1} G_1) \com \cdots \com (W_n\Reach{T_n} G_n)\big\}>p \Big \}. \nonumber
			\end{flalign}
		\end{subequations}	 	 
	\end{Def}


	\begin{Rem}[Stochastic Reach-Avoid Problem]
		The motion planning scenarios for only two sets $(W_1, G_1)$ basically reduce to the basic reach-avoid maneuver studied in our earlier work \cite{} by setting the reach set to $G_1$ and the avoid set to $\R^d \setminus W_1$. See also \cite{ref:GaoLygerosQuincampoix-2006} for the corresponding deterministic and \cite{ref:Summers-10} for the corresponding discrete time stochastic reach-avoid problems.
	\end{Rem}

	\begin{Rem}[Mixed Motion Planning Events]
	\label{rem:mix event}	
	One may also consider an event that consists of a mixture of the events in \eqref{event}, e.g., $(W_1 \Path G_1)_{\le T_1} \circ (W_2 \Reach{T_2} G_2)$. Following essentially the same analytical techniques as the ones proposed in subsequent sections, one can also address these mixed motion planning objectives; see \cite{ref:MohMilCha-2013} for an example of this nature.
	\end{Rem}

	\begin{Rem}[Time-varying Goal and Way Point Sets]
		In Definition \ref{def:initial set} the motion planning objective is introduced in terms of stationary (time-independent) goal and way point sets. However, note that one can always augment the state space with time, and introduce a new stochastic process $Y_t \Let [X_t^{\tr}, t]^{\tr}$. Therefore, a motion planning concerning moving sets for $X_t$ can be viewed as a motion planning with stationary sets for the process $Y_t$. 
	\end{Rem}

\section{Connection to Stochastic Optimal Control Problems} \label{sec:connection}

	In this section we establish a connection between stochastic motion planning initial condition sets $\PathSet$ and $\wt \PathSet$ of Definition \ref{def:initial set} and a class of stochastic optimal control problems involving stopping times. First, given a sequence of sets we introduce a sequence of random times that corresponds to the times that the process $\traj{X}{t,x}{\control{u}}{\cdot}$ exits each set in the sequence for the first time.

	\begin{Def}[Sequential Exit-Time]
	\label{def:theta}
		Given an initial condition $(t,x) \in \set{S}$ and a sequence of measurable sets $(A_i)_{i=k}^n \subset \borel(\R^d)$, the sequence of random times $\big(\Theta^{A_{k:n}}_i \big)_{i=k}^n $ defined\footnote{By convention, $\inf \emptyset = \infty$.} by
		\begin{align*}
		\begin{cases}\vspace{1mm}
			\Theta^{A_{k:n}}_i(t,x) \Let \inf \big\{ r\ge \Theta^{A_{k:n}}_{i-1}(t,x) ~:~ \traj{X}{t,x}{\control{u}}{r} \notin A_i \big\}, \qquad
			\\ \Theta^{A_{k:n}}_{k-1}(t,x) \Let t,
		\end{cases} 
		\end{align*}
		is called the \emph{sequential exit-time} through the set $A_k$ to $A_n$.
	\end{Def}

	Note that the sequential exit-time $\Theta^{A_{k:n}}_i$ depends on the control $\control{u}$ in addition to the initial condition $(t,x)$, but here and later in the sequel we shall suppress this dependence. For notational simplicity, we may also drop $(t,x)$ in subsequent sections. We refer the interested reader to Lemma \ref{fact:measurability} showing that these sequential stopping times are indeed well-defined.

	\begin{figure}[t!]
		\centering
				\centering		
				\includegraphics[scale= 0.4]{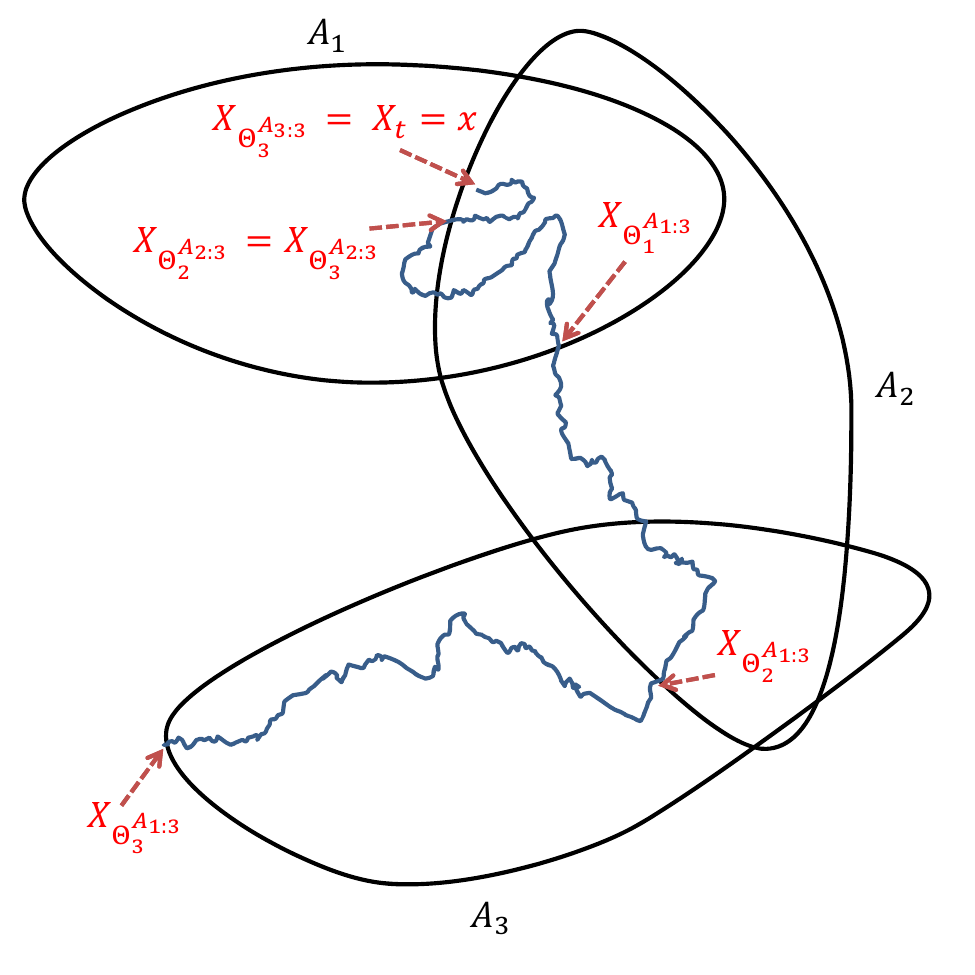}
				\caption{{\footnotesize Sequential exit-times of a sample path through the sets $(A_i)_{i=k}^3$ for different values of $k$}}
				\label{fig:theta_A}
	\end{figure}	

	In Figure \ref{fig:theta_A} a sample path of the process $\traj{X}{t,x}{\control{u}}{\cdot}$ along with the sequential exit-times $(\Theta_i^{A_{k:3}})_{i=k}^n$ is depicted for different $k \in \{1,2,3\}$. Note that since the initial condition $x$ does not belongs to $A_3$, the first exit-time of the set $A_3$ is indeed the start time $t$, i.e., $\Theta_3^{A_{3:3}} = t$. Let us highlight the difference between stopping times $\Theta_2^{A_{1:3}}$ and $\Theta_2^{A_{2:3}}$. The former is the first exit-time of the set $A_2$ after the time that the process leaves $A_1$, whereas the latter is the first exit-time of the set $A_2$ from the very beginning. {In Section \ref{sec:application} we shall see that these differences will lead to different definitions of value functions in order to derive a dynamic programming argument.}
	

	Given $(W_i,G_i, T_i)_{i=1}^n \subset \borel(\R^d)\times\borel(\R^d)\times [t,T]$, we introduce two value functions $V, \wt V:\set{S} \ra [0,1]$ defined by
	\begin{subequations}
	\label{V's}
		\begin{align}
		\label{V}	
		&\left \{
		\begin{array}{l}\vspace{3mm}
			V(t,x) \Let \sup\limits_{\control{u} \in \controlmaps{U}_t} \E{\prod\limits_{i = 1}^{n} \ind{G_i} \big(\traj{X}{t,x}{\control{u}}{\eta_i}\big)}, \\
			\eta_i \Let \Theta^{B_{1:n}}_i \mn T, \qquad B_i \Let W_i \setminus G_i,
		\end{array}
		\right.\\
		\label{Vtilde}	
		&\left \{
		\begin{array}{l}\vspace{3mm}
 		\wt V(t,x) \Let \sup\limits_{\control{u} \in \controlmaps{U}_t} \E{\prod\limits_{i = 1}^{n} \ind{G_{i} \cap W_i} \big(\traj{X}{t,x}{\control{u}}{\wt \eta_i}\big)}, \\
 		\wt \eta_i \Let \Theta^{W_{1:n}}_i \mn T_i,
		\end{array}
		\right.
		\end{align}
	\end{subequations}
	where $\Theta^{W_{1:n}}_i, \Theta^{B_{1:n}}_i$ are the sequential exit-times in the sense of Definition \ref{def:theta}. Figure \ref{fig:theta:path} and \ref{fig:theta:reach} illustrate the sequential exit-times corresponding to the sets $B_i$ and $W_i$, respectively. The main result of this section, Theorem \ref{thm:SOC} below, establishes a connection from the sets $\PathSet, \wt \PathSet$ and superlevel sets of the value functions $V$ and $\wt V$.
	
	\begin{figure}[t!]
	\centering
		\subfigure[Sequential exit-times related to motion planning event \eqref{path-event}]{\label{fig:theta:path}\includegraphics[scale = 0.39]{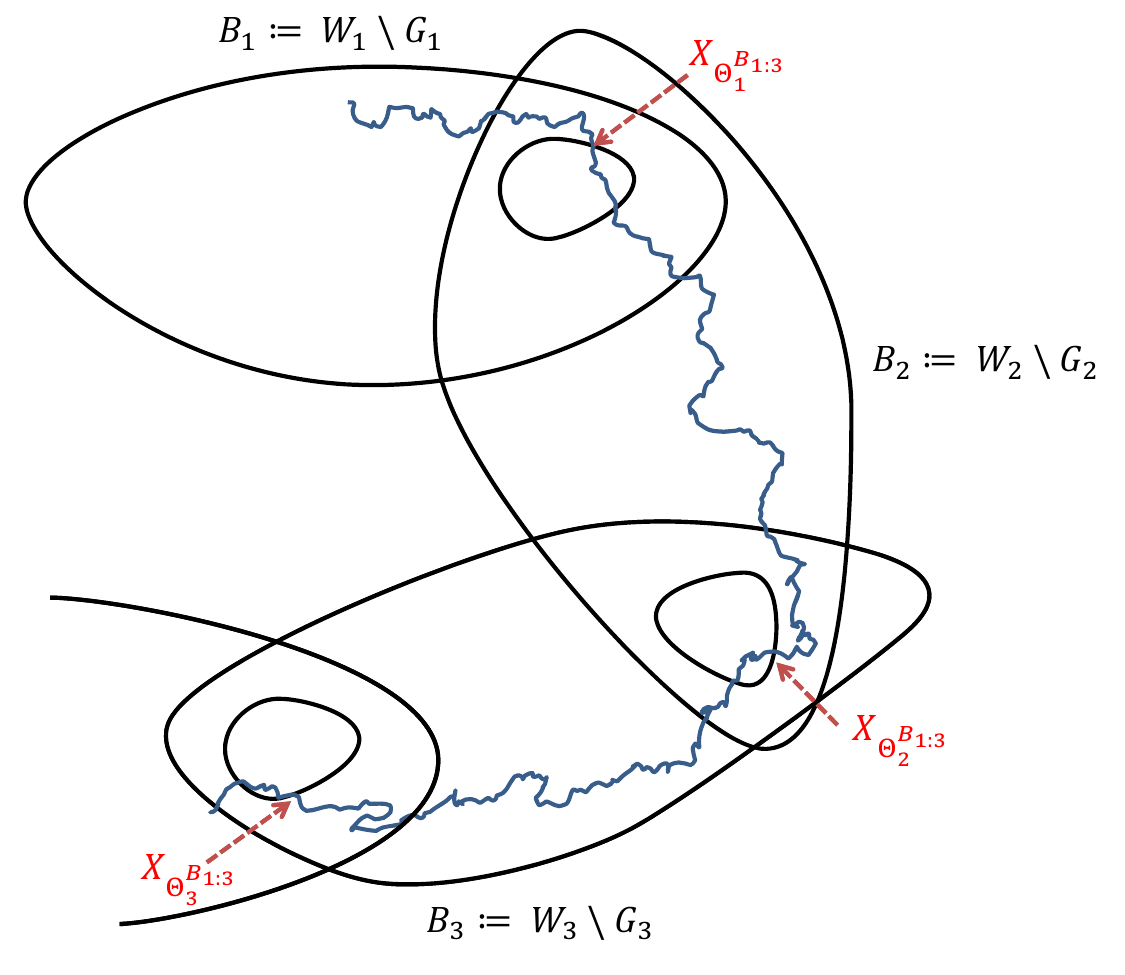}} 
		\quad
		\subfigure[Sequential exit-times related to motion planning event \eqref{reach-event}]{\label{fig:theta:reach}\includegraphics[scale = 0.39]{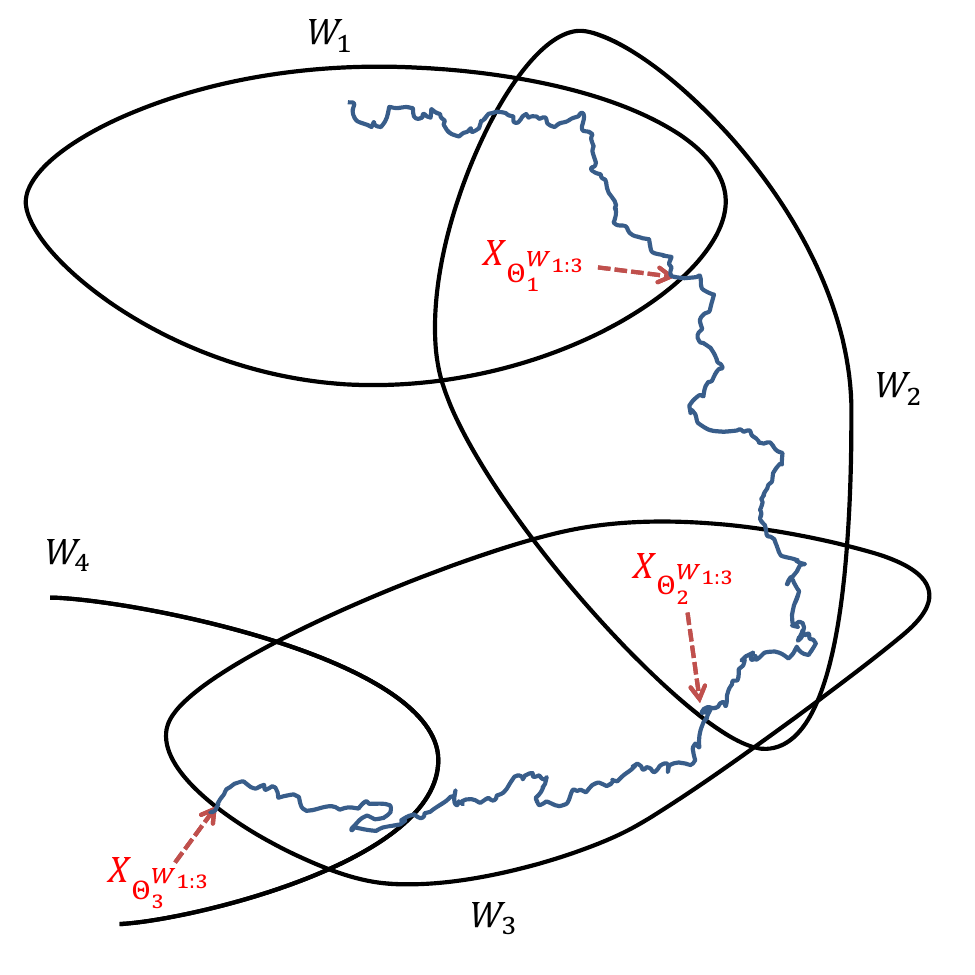}}
		\caption{Sequential exit-times corresponding to different motion planning events as introduced in \eqref{event}}
		\label{fig:theta}
	\end{figure}

	\begin{Thm}[Optimal Control Characterization]
	\label{thm:SOC}
		Fix a probability level $p \in [0,1]$, a sequence of set pairs $(W_i,G_i)_{i=1}^n \subset \borel(\R^d) \times \borel(\R^d)$, an initial time $t \in [0,T]$, and intermediate times $(T_i)_{i=1}^n \subset [t,T]$. Then,
		\begin{align}
		\label{prop V}  \PathSet \big(t,p ; (W_i,G_i)_{i=0}^n, T \big) & = \big\{ x\in \R^d ~ \big | ~ V(t,x) > p \big \}.	
		\end{align}
		Moreover, suppose Assumption \ref{a:set}.\ref{a:set:open} holds. Then,
		\begin{align}
		\label{prop Vtilde} \wt \PathSet \big(t,p ; (W_i,G_i)_{i=0}^n, (T_i)_{i=1}^n \big) & = \big\{ x\in \R^d ~ \big | ~ \wt V(t,x) > p \big \}, 
		\end{align}
		where the value functions $V$ and $\wt V$ are as defined in \eqref{V's}.
	\end{Thm}

	\begin{proof} See Appendix \ref{app-A}.\end{proof}
	
	Intuitively speaking, observe that the value functions \eqref{V's} consist of a sequence of indicator functions, where the reward is $1$ when the corresponding phase (i.e., reaching $G_i$ while staying in $W_i$) of motion planning is fulfilled, while the reward is $0$ if it fails. Let us also highlight that the difference between the time schedule between the two motion planning problems in \eqref{V's} is captured via the stopping times $\eta_i$ and $\wt{\eta}_i$: the former refers to the first time to leave $W_i$ or hit $G_i$ before $T$, and the latter only considers the exit time from $W_i$ prior to $T_i$. Hence, the product of the indicators evaluates to $1$ if and only if the entire journey comprising $n$ phases is successfully accomplished. In this light, taking expectations yields the probability of the desired event, and the supremum over admissible controls leads to the assertion that there exists a control for which the desired properties hold.
	

\section{Dynamic Programming Principle} \label{sec:DPP}

	The objective of this section is to derive a DPP for the value functions $V$ and $\wt V$ introduced in \eqref{V's}. The DPP provides a bridge between the theoretical characterization of the solution to our motion planning problem through value functions (Section \ref{sec:connection}) and explicit characterizations of these value functions using, for example, PDEs (Section \ref{sec:application}), which can then be used to solve the original problem numerically.
	
	Let $(T_i)_{i=1}^n \subset [0,T]$ be a sequence of times, $(A_i)_{i=1}^n \subset \borel(\R^d)$ be a sequence of open sets, and $\ell_i : \R^r \ra \R$ for $i = 1,\cdots,n$ be a sequence of measurable and bounded payoff functions. We define the sequence of value functions $V_k: [0,T] \times \R^d \ra \R^d$ for each \(k \in  \{1, \ldots, n\} \) as
	\begin{align}
		\label{Vk}
		&\left \{
		\begin{array}{l}\vspace{3mm}
		V_k(t,x) \Let \sup\limits_{\control{u} \in \controlmaps{U}_t} \EE \Big[ \prod\limits_{i=k}^n \ell_i\big(\traj{X}{t,x}{\control{u}}{\tau_i^k}\big)\Big], \\
		\tau_i^k(t,x) \Let \Theta^{A_{k:n}}_i(t,x) \mn T_i, \quad i \in \{k,\cdots,n\},
		\end{array}
		\right. 
	\end{align}
	where the stopping times $(\Theta^{W_{k:n}}_i)_{i=k}^n$ are sequential exit-times in the sense of Definition \ref{def:theta}. Recall that the sequential exit-times of $V_k$ correspond to an excursion through the sets $(A_i)_{i=k}^n$ irrespective of the first $(k-1)$ sets. It is straightforward to observe that the value function $V$ \big (resp.\ $\wt V$ \big ) in \eqref{V's} is a particular case of the value function $V_1$ defined as in \eqref{Vk} when $A_i \Let W_i \setminus G_i$ \big (resp.\ $A_i \Let W_i$ \big), $\ell_i \Let \ind{G_i}$ \big(resp.\ $\ell_i \Let \ind{G_i \cap W_i}$\big), and $T_i \Let T$. 
	
	To state the main result of this section, Theorem \ref{thm:DPP} below, some technical definitions and assumptions concerning the stochastic processes $\traj{X}{t,x}{\control{u}}{\cdot}$, the admissible controls $\controlmaps{U}_t$, and the payoff functions $\ell_i$ are needed:

	\begin{As}
	\label{a:DPP}
		For all $(t,x) \in \set{S}$, $\theta \in \setofst{t,T}$, and $\control{u}, \control{v} \in \controlmaps{U}_t$, we stipulate the following assumptions on 
		\begin{enumerate}[label=\alph*., itemsep = 2mm] 

			\item \textbf{Admissible controls:} \label{a:control} \\
				$\controlmaps{U}_t$ is the set of $\filtration_t$-progressively measurable processes taking values in a given control set. That is, the value of $\control{u} \Let (u_s)_{s \in [0,T]}$ at time $s$ is a measurable mapping $(z_{r \mx t} - z_t)_{r\in[0,s]} \mapsto u_s$ for all $s \in [0,T]$, see \cite[Def.\ 1.11, p.\ 4]{ref:KarShr-91} for more details.

			\item \textbf{Stochastic process:} \label{a:process}
			  	\begin{enumerate}[label=\roman*., itemsep = 2mm] 
					\item \label{a:process:causal} \emph{Causality:} 
						If $\ind{[t,\theta]} \control{u} = \ind{[t,\theta]}\control{v}$, then we have $\ind{[t,\theta]}\traj{X}{t,x}{\control{u}}{\cdot} = \ind{[t,\theta]} \traj{X}{t,x}{\control{v}}{\cdot}$. 
					
					\item \label{a:process:markov} \emph{Strong Markov property:} For each $\omega \in \Omega$ and the sample path $(z_r)_{r \in [0,\theta(\omega)]}$ up to the stopping time $\theta$, let the random control $\control{u}_\theta \in \controlmaps{U}_{\theta(\omega)}$ be the mapping 
					{\small $(z_{\cdot \mx \theta(\omega)} - z_{\theta(\omega)}) \mapsto \control{u}(z_{\cdot \mn \theta(\omega)} + z_{\cdot \mx \theta(\omega)} - z_{\theta(\omega)}) =: \control{u}_\theta$}.\footnote{Notice that $z_{\cdot} \equiv z_{\cdot \mn \theta(\omega)} + z_{\cdot \mx \theta(\omega)} - z_{\theta(\omega)}$. Thus, the randomness of $\control{u}_\theta$ is referred to the term $z_{.\mn \theta(\omega)}$.} Then, it holds with probability one that 
					\begin{align*} 
						\EE \Big[ & \ell \big( \traj{X}{t,x}{\control{u}}{\theta+s} \big) ~\Big|~ \sigalg_\theta \Big] = \EE \Big[ \ell\big( \traj{X}{\bar t,\bar x}{\bar {\control{u}}}{\theta+s} \big) ~ \Big|~ \bar t = \theta, \bar x = \traj{X}{t,x}{\control{u}}{\theta}, \bar {\control{u}} = \control{u}_\theta \Big ], 
					\end{align*} 					
					for all bounded measurable functions $\ell : \R^d \ra \R$ and $s\ge 0$. 
					
					\item  \label{a:process:exit-time} \emph{Continuity of the exit-times:} 
						Given initial condition $(t_0,x_0) \in \set{S}$, for all $k \in \{1,\cdots,n\}$ and $i \in \{ k, \cdots, n\}$ the stochastic mapping $(t,x) \mapsto \traj{X}{t,x}{\control{u}}{\tau_i^k(t,x)}$ is $\PP$-a.s.\ continuous at $(t_0,x_0)$ where the stopping times $\tau_i^k$ are defined as in \eqref{Vk}.
			  	\end{enumerate}		  
			  	
			\item \textbf{Payoff functions:} \label{a:payoff} \\
				$(\ell_i)_{i=1}^n$ are lower semicontinuous for all $i \le n$.
		\end{enumerate}	
	\end{As}

	\begin{Rem}
	\label{rem:DPP}
		Some remarks on the above assumptions are in order:
		\begin{itemize}[label=$\circ$, itemsep = 2mm]
			\item Assumption \ref{a:DPP}.\ref{a:control} implies that the admissible controls  $\control{u} \in \controlmaps{U}_t$ take action at time $t$ independent of future information arriving at $s>t$. This is known as the \emph{non-anticipative} strategy and is a standard assumption in stochastic optimal control \cite{ref:Borkar-05}.
		
			\item Assumption \ref{a:DPP}.\ref{a:process} imposes three constraints on the process $\traj{X}{t,x}{\control{u}}{\cdot}$ defined on the prescribed probability space: {i)} causality of the solution processes for a given admissible control {ii)} strong Markov property {iii)} continuity of exit-time. The causality property is always satisfied in practical applications; uniqueness of the solution process $\traj{X}{t,x}{\control{u}}{\cdot}$ under any admissible control process $\control{u}$ guarantees it. The class of Strong Markov processes is fairly large; for instance, it contains the solution of SDEs under some mild assumptions on the drift and diffusion terms \cite[Thm.\ 2.9.4]{Krylov_ControlledDiffusionProcesses}. The almost sure continuity of the exit-time with respect to the initial condition of the process is the only potentially restrictive part of the assumptions. Note that this condition does not always hold even for deterministic processes with continuous trajectories. One may need to impose conditions on the process and possibly the sets involved in motion planning in order to satisfy continuity of the mapping $(t,x) \mapsto \traj{X}{t,x}{\control{u}}{\tau_i^k(t,x)}$ at the given initial condition with probability one. We shall elaborate on this issue and its ramifications for a class of diffusion processes in Section \ref{sec:application}.
			
			\item Assumption \ref{a:DPP}.\ref{a:payoff} imposes a fairly standard assumption on the payoff functions. In case $\ell_i$ is the indicator function of a given set, for example in \eqref{V's}, this assumption requires the set to be open. This issue will be addressed in more details in Subsection \ref{subsec:numerical}, in particular to bring a reconciliation with Assumption \ref{a:set}.\ref{a:set:closed}.
		\end{itemize}
	\end{Rem}

	Let function $J_k : \set{S}\times \controlmaps{U}_0 \ra \R$ be
	\begin{equation}
	\label{J}
		J_k(t,x; \control{u}) \Let \EE \Big[ \prod_{i=k}^n \ell_i\big(\traj{X}{t,x}{\control{u}}{\tau_i^k}\big)\Big], 
	\end{equation}
	where $\big(\tau_i^k\big)_{i = k}^n$ are as defined in \eqref{Vk}. 
	

	The following Theorem, the main result of this section, establishes a dynamic programming argument for the value function $V_k$ in terms of the ``successor'' value functions $(V_j)_{j = k+1}^{n}$, all defined as in \eqref{Vk}. For ease of notation, we shall introduce deterministic times $\tau^k_{k-1}, \tau^k_{n+1}$, and a trivial constant value function $V_{n+1}$.

	\begin{Thm}[Dynamic Programming Principle]
	\label{thm:DPP}
		Consider the value functions $(V_j)_{j = 1}^n$ and the sequential stopping times $( \tau_j^k )_{j=k}^n$ introduced in \eqref{Vk} where $k \in \{1,\cdots,n\}$. Under Assumptions \ref{a:DPP}, for all $(t,x) \in \set{S}$ and family of stopping times $\{ \theta^{\control{u}}, \control{u} \in \controlmaps{U}_t\} \subset \setofst{t,T}$, we have 
		
		\vspace{-4mm}
		\begin{small}
		\begin{subequations}		
		\label{DPP}			
			\begin{align}			
				\label{DPP-sup} V_k(t,x) & \le 
				\sup_{\control{u} \in \controlmaps{U}_t} \EE \bigg[ \sum_{j = k}^{n+1} \ind{\{\tau_{j-1}^k \le \theta^{\control{u}} < \tau_j^k\}} {V^*_j} \big(\theta^{\control{u}}, \traj{X}{t,x}{\control{u}}{\theta^{\control{u}}}\big) \prod_{i = k}^{j-1}\ell_i\big( \traj{X}{t,x}{\control{u}}{\tau_i^k}\big)   \bigg] ,\\
				\label{DPP-sub} V_k(t,x)  & \ge  
				\sup_{\control{u} \in \controlmaps{U}_t} \EE \bigg[ \sum_{j = k}^{n+1} \ind{\{\tau_{j-1}^k \le \theta^{\control{u}} < \tau_j^k \}} {V_j}_* \big(\theta^{\control{u}}, \traj{X}{t,x}{\control{u}}{\theta^{\control{u}}}\big) \prod_{i = k}^{j-1}\ell_i\big( \traj{X}{t,x}{\control{u}}{\tau_i^k}\big) \bigg] , 
			\end{align}
		\end{subequations}
		\end{small}
		where ${V^*_j}$ and ${V_j}_*$ are, respectively, the upper and the lower semicontinuous envelope of $V_j$, $\tau_{k-1}^k \Let t$, $V_{n+1} \equiv 1$, and $\tau_{n+1}^k$ is any constant time strictly greater than $T$, say $\tau_{n+1}^k \Let T+1$.			
	\end{Thm}

\begin{proof} See Appendix \ref{app-A}. \end{proof}

	In our context the DPP proposed in Theorem \ref{thm:DPP} allows us to characterize the value function \eqref{Vk} through a sequence of value functions $(V_j)_{j = k+1}^{n}$. That is, \eqref{DPP-sup} and \eqref{DPP-sub} impose mutual constraints on a value function and subsequent value functions in the sequence. Moreover, the last function in the sequence is fixed to a constant by construction. Therefore, assuming that an algorithm to sequentially solve for these mutual constraints can be established, one could in principle use it to compute all value functions in the sequence and solve the original motion planning problem. In Section \ref{sec:application} we show how, for a class of controlled diffusion processes, the constraints imposed by \eqref{DPP} reduce to PDEs that the value functions need to satisfy. This enables the use of numerical PDE solution algorithms for this purpose.
	
	\begin{Rem}[Semicontinuity and Measurability]
	\label{rem:measurability}
		Theorem \ref{thm:DPP} introduces DPP's in the spirit of \cite{BouchardTouzi_WeakDPP} but in an exit-time framework, which is in a weaker sense than the standard DPP in stochastic optimal control problems \cite[Section IV.7]{SonerBook}. Namely, the DPP only requires a semicontinuity of the payoff functions while it avoids the verification of the measurability of the value functions $V_k$ in \eqref{V's}. Notice that in general this measurability issue is non-trivial due to the supremum operation running over possibly uncountably many controls.
	\end{Rem}

\section{The Case of Controlled Diffusions} \label{sec:application}
	
	In this section we come to the last step in our construction. We demonstrate how the DPP derived in Section \ref{sec:DPP}, in the context of controlled diffusion processes, gives rise to a series of PDE's. Each PDE is understood in the discontinuous viscosity sense with boundary conditions in both Dirichlet (pointwise) and viscosity senses. This paves the way for using PDE numerical solvers to numerically approximate the solution of our original motion planning problem for specific examples. We demonstrate an instance of such an example in Section \ref{sec:simulation}. 
	
	We first introduce formally the standard probability space setup for SDEs, then proceed with some preliminaries to ensure that the requirements of the proposed DPP, Assumptions \ref{a:DPP}, hold. The section consists of subsections concerning PDE derivation and boundary conditions along with further discussions on how to deploy existing PDE solvers to numerically compute our PDE characterization. 

	Let $\Omega$ be $\mathcal{C}\big([0,T],\R ^{z_d} \big)$, the set of continuous functions from $[0,T]$ to $\R^{z_d}$, and let $(z_t)_{t \ge 0}$ be the canonical process, i.e., $z_t(\omega) \Let \omega_t $. We consider $\PP$ as the Wiener measure on the filtered probability space $(\Omega, \sigalg, \filtration)$, where $\filtration$ is the smallest right continuous filtration on $\Omega$ to which the process $(z_t)_{t\ge0}$ is adapted. Let us recall that $\filtration_t \Let (\sigalg_{t,s})_{s\ge0} $ is the auxiliary subfiltration defined as $\sigalg_{t,s} \Let \sigma\big(z_{r \mx t} - z_t, r \in [0,s]\big)$. Let $\set{U}\subset \R^{d_u}$ be a control set, and $\controlmaps{U}_t$ denote the set of all $\filtration_t$- progressively measurable mappings into $\set{U}$. For every $\control{u} = (u_t)_{t \ge 0}$ we consider the $\R^d$-valued SDE\footnote{We slightly abuse notation and earlier used $\sigma$ for the sigma algebra as well. However, it will be always clear from the context to which $\sigma$ we refer.}
	\begin{equation}
	\label{SDE}
		\diff X_s = f(X_s, u_s)\,\diff s + \sigma(X_s, u_s)\,\diff W_s,\qquad X_t = x, 
	\end{equation}
	where $f:\R^d\times \set{U}\ra\R^d$ and $\sigma:\R^d\times\set{U}\ra\R^{d \times d_z}$ are measurable functions, and $W_s \Let z_s$ is the canonical process.
	
	\begin{As}
	\label{a:SDE}
		We stipulate that
		\begin{enumerate}[label=\alph*.] 
			\item \label{a:SDE:compact} $\set{U}\subset\R^m$ is compact;
			\item \label{a:SDE:lip} $f$ and $\sigma$ are continuous and Lipschitz in first argument uniformly with respect to the second;
			\item \label{a:SDE:nondegenerate} The diffusion term $\sigma$ of the SDE (\ref{SDE}) is uniformly non-degenerate, i.e., there exists $\delta>0$ such that for all $x \in \R^{d}$ and $u \in \set{U}$, $\| \sigma(x,u) \sigma^{\top}(x,u) \| > \delta$.
		\end{enumerate}
	\end{As}
	It is well-known that under Assumptions \ref{a:SDE}.\ref{a:SDE:compact} and \ref{a:SDE}.\ref{a:SDE:lip} there exits a unique strong solution to the SDE \eqref{SDE} \cite{ref:Borkar-05}; let us denote it by $\big(\traj{X}{t,x}{\control{u}}{s}\big)_{s \ge t}$. For future notational simplicity, we slightly modify the definition of $\traj{X}{t, x}{\control{u}}{s}$, and extend it to the whole interval $[0,T]$ where $\traj{X}{t, x}{\control{u}}{s} \Let x$ for all $s$ in $ [0, t]$.

	In addition to Assumptions \ref{a:SDE} on the SDE \eqref{SDE}, we impose the following assumption on the motion planning sets that allows us to guarantee the continuity of sequential exit-times, as required for the DPP obtained in the preceding section.

	\begin{As} 
	\label{a:SDE:set}
	The open sets $(A_i)_{i=1}^n$ satisfy the exterior cone condition, i.e., for every $i \in \{1,\cdots, n\}$, there are positive constants $h$, $r$ an $\R^d$-value bounded map $\eta: A_i^c \ra \R^d$ such that
	\begin{equation*}
		\ball{B}_{rt}( x+\eta(x)t) \subset A_i^c \qquad \text{for all $x \in A_i^c$ and $t \in (0,h]$ }
	\end{equation*}
	where $\ball{B}_r(x)$ denotes an open ball centered at $x$ and radius $r$ and $A_i^c$ stands for the complement of the set $A_i$.
	\end{As}
		
	\begin{Rem}[Smooth Boundary]
	If the set $A_i$ is bounded and its boundary $\partial A_i$ is smooth, then Assumption \ref{a:SDE:set} holds. Furthermore, boundaries with corners may also satisfy Assumption \ref{a:SDE:set}; Figure \ref{fig:boundary} depicts two different examples.
	\end{Rem}	
	\begin{figure}[t!]
		\centering
		\subfigure[Exterior cone condition holds at every point of the boundary.]{\label{fig:ExtCone}\includegraphics[scale = 0.5]{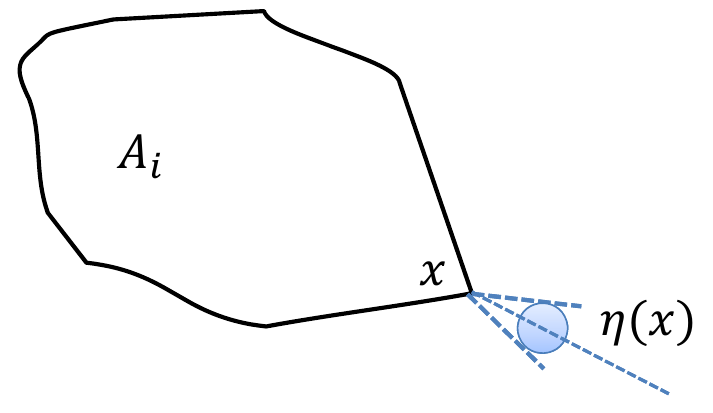}} 
		\qquad
		\subfigure[exterior cone condition fails at the point $x$---the only possible exterior cone at $x$ is a line.]{\label{fig:NoIntCone}\includegraphics[scale = 0.5]{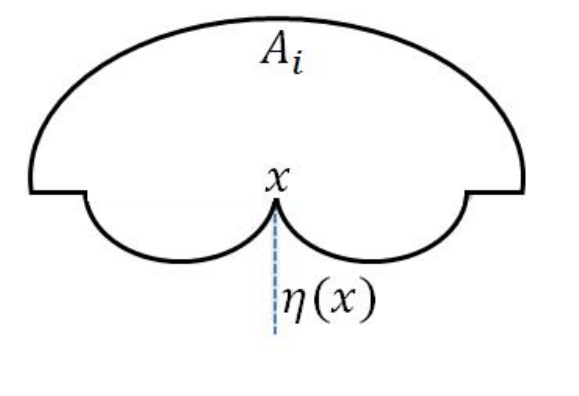}}
		\caption{Exterior cone condition of the boundary}
		\label{fig:boundary}
	\end{figure}	
	
\subsection{Sequential PDEs}
\label{subsec:PDE}
	In the context of SDEs, we show how the abstract DPP of Theorem \ref{thm:DPP} results in a sequence of PDEs, to be interpreted in the sense of discontinuous viscosity solutions; for the general theory of viscosity solutions we refer to \cite{Crandali_Ishii_Lions_VoscositySolutions} and \cite{SonerBook}. For numerical solutions to these PDEs, one also needs appropriate boundary conditions which is addressed in the next subsection. 
	
	To apply the proposed DPP, one has to make sure that Assumptions \ref{a:DPP} are satisfied. As pointed out in Remark \ref{rem:DPP}, the only nontrivial assumption in the context of SDEs is Assumption \ref{a:DPP}.\ref{a:process}\ref{a:process:exit-time} The following proposition addresses this issue, and allows us to employ the DPP of Theorem \ref{thm:DPP} for the main result of this subsection. 

	\begin{Prop}[Stopped-Process Continuity]
	\label{prop:exit-time continuity}
		Consider the SDE \eqref{SDE} where Assumptions \ref{a:SDE} hold. Suppose the open sets $(A_i)_{i=1}^n \subset \borel(\R^d)$ satisfy the exterior cone condition in Assumption \ref{a:SDE:set}. Let $\big(\Theta^{A_{1:n}}_i \big)_{i=1}^n $ be the respective sequential exit-times as defined in Definition \ref{def:theta}. Given intermediate times $(T_i)_{i=1}^n$, and control $\control{u}\in \controlmaps{U}_t$, consider the stopping time $\tau_i \Let \Theta^{A_{1:n}}_i \mn T_i$. Then, for any $i \in \{1.\cdots,n\}$, initial condition $(t,x) \in \set{S}$, and sequence of initial conditions $(t_m,x_m) \ra (t,x)$, we have
		\begin{align*}
		 	\lim_{m \ra \infty} & \tau_i(t_m,x_m) = \tau_i(t,x) \quad \PP \text{-a.s.}, 
		\end{align*}
		As a consequence, the stochastic mapping $(t,x) \mapsto \traj{X}{t,x}{\control{u}}{\tau_i(t,x)}$ is continuous with probability one, i.e., $\lim\limits_{m \ra \infty} \traj{X}{t_m,x_m}{\control{u}}{\tau_i(t_m,x_m)} = \traj{X}{t,x}{\control{u}}{\tau_i(t,x)}$ $\PP$-a.s.\ for all $i$.
	\end{Prop}
	
	\begin{proof} See Appendix \ref{app-B}.\end{proof} 
	
	\begin{Def}[Dynkin Operator]
	\label{def:Dynkin operator}
		Given $u\in \set{U}$, we denote by $\mathcal{L}^u$ the \emph{Dynkin} operator (also known as the infinitesimal generator \cite[Chap.\ 5]{ref:KarShr-91}) associated to the SDE \eqref{SDE} as
			\begin{align*}
				\mathcal{L}^u \Phi(t,x) := \partial_t \Phi(t,x) & + f(x,u).\partial_x \Phi(t,x) + \frac{1}{2}\text{Tr}[\sigma(x,u) \sigma^\top(x,u) \partial_{x}^2 \Phi(t,x)],
			\end{align*}
		where $\Phi$ is a real-valued function smooth on the interior of $\set{S}$, with $\partial_t \Phi$ and $\partial_x \Phi$ denoting the partial derivatives with respect to $t$ and $x$, respectively, and $\partial^2_x \Phi$ denoting the Hessian matrix with respect to $x$. 
	\end{Def}

	Theorem \ref{thm:DPE} is the main result of this subsection, which provides a characterization of the value functions $V_k$ in terms of Dynkin operator in Definition \ref{def:Dynkin operator} in the interior of the set of interest, i.e., $[0,T_k[ \times A_k$. We refer to \cite[Thm.\ 17.23]{ref:Kallenberg-97} for details on the above differential operator. 
	
	\begin{Thm}[Dynamic Programming Equation]
	\label{thm:DPE}
		Consider the system \eqref{SDE}, and suppose that Assumptions \ref{a:SDE} hold. Let the value functions $V_k : \set{S} \ra \R^d$ be as defined in \eqref{Vk}, where the sets $(A_i)_{i=1}^n$ satisfy Assumption \ref{a:SDE:set}, and the payoff functions $(\ell_i)_{i=1}^n$ are all lower semicontinuous. Then,
		\begin{itemize}[label=$\circ$] 
			\item ${V_k}_*$ is a viscosity supersolution of
				\begin{equation*}
				\label{supersolution V1}
					-\sup_{u \in \set{U}} \mathcal{L}^u {V_k}_*(t,x) \geq 0  \qquad \text{on} \quad [0,T_k[ \times A_k;
				\end{equation*}
			\item ${V^*_k}$ is a viscosity subsolution of
				\begin{equation*}
				\label{supersolution}
					-\sup_{u \in \set{U}} \mathcal{L}^u {V^*_k}(t,x) \leq 0  \qquad \text{on} \quad [0,T_k[ \times A_k.
				\end{equation*}
		\end{itemize}
	\end{Thm}
	
	\begin{proof}
		See Appendix \ref{app-B}.  
	\end{proof}
	
\subsection{Boundary Conditions}
\label{subsec:boundary}
	To numerically solve the PDE of Theorem \ref{thm:DPE}, one needs boundary conditions on the complement of the set where the PDE is defined. This requirement is addressed in the following proposition.
	
	\begin{Prop}[Boundary Conditions]
	\label{prop:boundary}
		Suppose that the hypotheses of Theorem \ref{thm:DPE} hold. Then the value functions $V_k$ introduced in \eqref{Vk} satisfy the following boundary value conditions:
		
	\vspace{-4mm}
	\begin{small}
		\begin{subequations}
		\label{boundary}		
			\begin{align}
			\label{boundary pointwise} 
			\text{Dirichlet:}&
			\begin{cases}  
				{V_k}(t,x) = {V_{k+1}}(t,x){\ell_{k}(x)} \hspace{17mm} \vspace{1mm} \\ 
				\forall (t,x) \in [0,T_k] \times  A_k^c \bigcup \{T_k\} \times \R^d 
			\end{cases}\\
			\label{boundary visc} 
			\text{Viscosity:}&
			\begin{cases}
				\limsup\limits_{\footnotesize \begin{smallmatrix} A_k \ni x' \ra x \\ t' \ua t \end{smallmatrix}} {V_k}(t',x') \le  V^*_{k+1}(t,x)\ell^*_{k}(x) 
						\\
				\liminf\limits_{\footnotesize \begin{smallmatrix} A_k \ni x' \ra x \\ t' \ua t
				\end{smallmatrix}}{V_k}(t',x') \ge  {V_{k+1}}_*(t,x){\ell_k}(x) \\
				\forall (t,x) \in [0,T_k] \times  \partial A_k \bigcup \{T_k\} \times \ol{A}_k
			\end{cases}
			\end{align}
		\end{subequations}
	\end{small}
	\end{Prop}

	\begin{proof} See Appendix \ref{app-B}.\end{proof}

	Proposition \ref{prop:boundary} provides boundary condition for $V_k$ in both Dirichlet (pointwise) and viscosity senses. The Dirichlet boundary condition \eqref{boundary pointwise} is the one usually employed to numerically compute the solution via PDE solvers, whereas the viscosity boundary condition \eqref{boundary visc} is required for theoretical support of the numerical schemes and comparison results. 
	
	\begin{Rem}[Lower Semicontinuity]
		\label{rem:lsc}
		In the SDE setting, one can, without loss of generality, extend the class of admissible controls in the definition of $V_k$ to $\controlmaps{U}_0$, i.e., $V_k(t,x) = \sup_{\control{u} \in \controlmaps{U}_0} J_k(t,x;\control{u})$; for a rigorous technology to prove this assertion see \cite[Remark 5.2]{BouchardTouzi_WeakDPP}. Thus, $V_k$ is lower semicontinuous as it is a supremum over a fixed family of lower semicontinuous functions, see Lemma \ref{lem:lsc} in Appendix \ref{app-A}. In this light, one may argue that in the viscosity boundary condition \eqref{boundary visc}, the second assertion is subsumed by the Dirichlet boundary condition \eqref{boundary pointwise}.
	\end{Rem}

\subsection{Discussion on Numerical Issues}
\label{subsec:numerical}
	For the class of controlled diffusion processes \eqref{SDE}, Subsection \ref{subsec:PDE} developed a PDE characterization of the value function $V_k$ within the set $[0, T_k[ \times A_k$ along with boundary conditions in terms of the successor value function $V_{k+1}$ provided in Subsection \ref{subsec:boundary}. Since $V_{n+1} \equiv 1$, one can infer that Theorem \ref{thm:DPE} and Proposition \ref{prop:boundary} provide a series of PDE where the last one has known boundary condition, while the boundary conditions of earlier in the sequence are determined by the solution of subsequent PDE, i.e., $V_{k+1}$ provides boundary conditions for the PDE corresponding to the value function $V_k$. Let us highlight once again that the basic motion planning maneuver involving only two sets is effectively the same as the first step of this series of PDEs and was studied in our earlier work \cite{ref:MohChatLyg-15}. 
	
	Before proceeding with numerical solutions, we need to properly address two technical concerns:

	\begin{enumerate}[label=(\roman*), itemsep = 2mm] 
		\item On the one hand, for the definition \eqref{path-set} we need to assume that the goal set $G_i$ is closed so as to allow continuous transition into $G_i$; see Assumption \ref{a:set}.\ref{a:set:closed} and the following discussion. On the other hand, in order to invoke the DPP argument of Section \ref{sec:DPP} and its consequent PDE in Subsection \ref{subsec:PDE}, we need to impose that the payoff functions $(\ell_i)_{i=1}^n$ are all lower semicontinuous; see Assumption \ref{a:DPP}.\ref{a:payoff} In the case of the value function $V$ in \eqref{V}, this constraint results in $(G_i)_{i=1}^n$ all being open, which in general contradicts Assumption \ref{a:set}.\ref{a:set:closed}.

		\item Most of the existing PDE solvers provide theoretical guarantees for \emph{continuous} viscosity solutions, e.g., \cite{ref:Mitchell-toolbox}. Theorem \ref{thm:DPE}, on the other hand, characterizes the solution to the motion planning problem in terms of \emph{discontinuous} viscosity solutions. Therefore, it is a natural question whether we could employ any of available numerical methods to approximate the solution of our desired value function. 
	\end{enumerate}

	Let us initially highlight the following points: Concerning (i) it should be mentioned that this contradiction is not applicable for the motion planning initial set \eqref{reach-set} since the goal set $G_i$ can be simply chosen to be open without confining the continuous transitions. Concerning (ii), we would like to stress that this discontinuous formulation is inevitable since the value functions defined in \eqref{V's} are in general discontinuous, and any PDE approach has to rely on discontinuous versions.

	To address the above concerns, we propose an $\eps$-conservative but precise way of characterizing the motion planning initial set. Given $(W_i,G_i) \in \borel(\R^d) \times \borel(\R^d)$, let us construct a smaller goal set $G^\eps_i \subset G_i$ such that $G^\eps_i \Let \{ x \in G_i | \dist(x, G_i^c) \ge \eps\}$.\footnote{$\dist(x,A) \Let \inf_{y \in A}\|x - y\|$, where $\|\cdot\|$ stands for the Euclidean norm.} For sufficiently small $\eps>0$ one may observe that $W_i \setminus G^\eps_i$ satisfies Assumption \ref{a:SDE:set}. Note that this is always possible if $W_i \setminus G_i$ satisfies Assumption \ref{a:SDE:set} since one can simply take $\eps < h/2$, where $h$ is as defined in Assumption \ref{a:SDE:set}. Figure \ref{fig:eps} depicts this situation.
	\begin{figure}[t!]
		\centering
		\includegraphics[scale = 0.5]{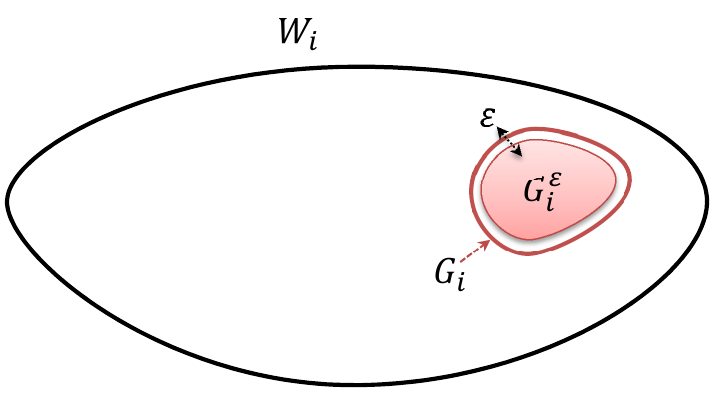}
		\caption{Construction of the sets $G^\eps_i$ from $G_i$ as described in Subsection \ref{subsec:numerical}}
		\label{fig:eps}
	\end{figure}

	Formally we define the payoff function $\ell^\eps_i : \R^d \ra \R$ as follows: 
	\begin{align*}
		\ell^\eps_i(x) \Let\biggl( 1 - \frac{\dist(x,G_i^\eps)}{\eps}\biggr) \mx 0 .
	\end{align*}
	Replacing the goal sets $G^\eps_i$ and payoff functions $\ell^\eps_i$ in \eqref{V}, we arrive at the value function
	\begin{align*}
		V^\eps(t,x) & \Let \sup_{\control{u} \in \controlmaps{U}_t} \E{\prod_{i = 1}^{n} \ell^\eps_i \big(\traj{X}{t,x}{\control{u}}{\eta^\eps_i}\big)}, \quad 
		\begin{cases}
			 \eta^\eps_i & \Let \Theta^{B^\eps_{1:n}}_i \mn T, \\ B^\eps_i & \Let W_i \setminus G^\eps_i.
		\end{cases}
	\end{align*}	
	It is straightforward to inspect that $V^\eps \le V$ since $G_i^\eps \subset G_i$. Moreover, with a similar technique as in \cite[Thm.\ 5.1]{ref:MohChatLyg-15}, one may show that $V(t,x) = \lim_{\eps \da 0} V^\eps(t,x)$ on the set $(t,x) \in [t,T[ \times \R^d$, which indicates that the approximation scheme can be arbitrarily precise. Note that the approximated payoff functions $\ell^\eps_i$ are, by construction, Lipschitz continuous that in light of uniform continuity of the process, Lemma \ref{lem:exit-time unif cont} in Appendix \ref{app-B}, leads to the continuity of the value function $V^\eps$.\footnote{This continuity result can, alternatively, be deduced via the comparison result of the viscosity characterization of Theorem \ref{thm:DPE} together with boundary conditions \eqref{boundary visc} \cite{Crandali_Ishii_Lions_VoscositySolutions}.} Hence, the discontinuous PDE characterization of Subsection \ref{subsec:PDE} can be approximated arbitrarily closely in the continuous regime.
	
	Let us recall that having reduced the motion planning problems to PDEs, numerical methods and computational algorithms exist to approximate its solution \cite{ref:Mitchell-toolbox}. In Section \ref{sec:simulation} we demonstrate how to use such methods to address practically relevant problems. In practice, such methods are effective for systems of relatively small dimension due to the curse of dimensionality. To alleviate this difficulty and extend the method to large problems, we can leverage on ADP \cite{ref:FarVan-03,ref:CogRotVanLall-06} or other advances in numerical mathematics, such as tensor trains \cite{ref:KhoSch-11}. The link between motion planning and the PDEs through DPP is precisely what allows us to capitalize on any such developments in the numerics.
	
\section{Numerical Example: Chemical Langevin Equation for a Biological Switch} 
\label{sec:simulation}

	When modeling uncertainty in biochemical reactions, one often resorts to countable Markov chain models \cite{ref:Wilkinson-06} which describe the evolution of molecular numbers. Due to the Markov property of chemical reactions, one can track the time evolution of the probability distribution for molecular populations as a family of ordinary differential equations called the \emph{chemical master equation} (CME) \cite{ref:Andrews-09,ref:Khammash&Gillespie-05}, also known as the forward Kolmogorov equation. 
	
	Though close to the physical reality, the CME is particularly difficult to work with analytically. One therefore typically employs different approximate solution methods, for example the Finite State Projection method \cite{ref:KhammashBook-10} or the moment closure method \cite{ref:Hespanha}. Such approximation method resorts to approximating discrete molecule numbers by a continuum and capturing the stochasticity in their evolution through a stochastic differential equation. This stochastic continuous-time approximation is called the \emph{chemical Langevin equation} or the \emph{diffusion approximation}, see for example \cite{ref:KhammashBook-10} and the references therein. The Langevin approximation can be inaccurate for chemical species with low copy numbers; it may even assign a negative number to some molecular species. To circumvent this issue we assume here that the species of interest come in sufficiently high copy numbers to make the Langevin approximation reasonable.
	
	 Multistable biological systems are often encountered in nature \cite{ref:Becskei01}. In this section we consider the following chemical Langevin formulation of a bistable two gene network: 
	\begin{align}
	\label{sde switch}
		\begin{cases}
		\diff X_t = \big ( f(Y_t, \control u_x) - \mu_x X_t \big)\diff t  + \sqrt{f(Y_t, \control{u}_x)}\diff W^1_t  + \sqrt{\mu_x X_t} \diff W^2_t, \\ 
			\diff Y_t = \big ( g(X_t, \control{u}_y) - \mu_y Y_t \big)\diff t + \sqrt{g(X_t, \control{u}_y)}\diff W^3_t +  \sqrt{ \mu_y Y_t} \diff W^4_t, 
		\end{cases}
	\end{align}
	where $X_t$ and $Y_t$ are the concentration of the two repressor proteins with the respective degradation rates $\mu_x$ and $\mu_y$; $(W^i_t)_{t\ge 0}$ are independent standard Brownian motion processes. Functions $f$ and $g$ are repression functions that describe the impact of each protein on the other's rate of synthesis controlled via some external inputs $\control{u}_x$ and $\control{u}_y$. 

	In the absence of exogenous control signals, the authors of \cite{ref:Cherry-00} study sufficient conditions on the drifts $f$ and $g$ under which the system dynamic \eqref{sde switch} without the diffusion term has two (or more) stable equilibria. In this case, system \eqref{sde switch} can be viewed as a biological switch network. The theoretical results of \cite{ref:Cherry-00} are also experimentally investigated in \cite{ref:Gardner-00} for a genetic toggle switch in \emph{Escherichia coli}. 

	Here we consider the biological switch dynamics where the production rates of proteins are influenced by external control signals; experimental constructs that can be used to provide such inputs have recently been reported in the literature \cite{ref:Milias-11}. The level of repression is described by a \emph{Hill} function, which models cooperativity of binding as follows:
	\begin{align*}
		f(y,u) \Let \frac{\theta^{n_1}_1k_1}{y^{n_1}+\theta^{n_1}_1} u, \qquad g(x,u) \Let \frac{\theta^{n_2}_2k_2}{x^{n_2}+\theta^{n_2}_2} u,
	\end{align*}
	where $\theta_i$ are the threshold of the production rate with respective exponents $n_i$, and $k_i$ are the production scaling factors. The parameter $u$ represents the role of external signals that affect the production rates, for which the control sets are $\set{U}_x \Let [\ul{u}_x, \ol{u}_x]$ and $\set{U}_y \Let [\ul{u}_y, \ol{u}_y]$. 
	In this example we consider system \eqref{sde switch} with the following parameters: $\theta_i = 40$, $\mu_i = 0.04$, $k_i = 4$ for both $i \in \{1,2\}$, and exponents $n_1 = 4$, $n_2 = 6$.	Figure \ref{fig:nullcline} depicts the drift nullclines and the equilibria of the system. The equilibria $z_a$ and $z_c$ are stable, while $z_b$ is the unstable one. We should remark that the ``stable equilibrium'' of SDE \eqref{sde switch} is understood in the absence of the diffusion term as the noise may very well push the states from one stable equilibrium to another.
	
	\begin{figure}[t!]
	\centering
		\subfigure[Nullclines and equilibria of the drift of the SDE \eqref{sde switch}; $x_a$ and $x_b$ are stable, and $x_c$ is unstable.]{\label{fig:nullcline}\includegraphics[scale = 0.2]{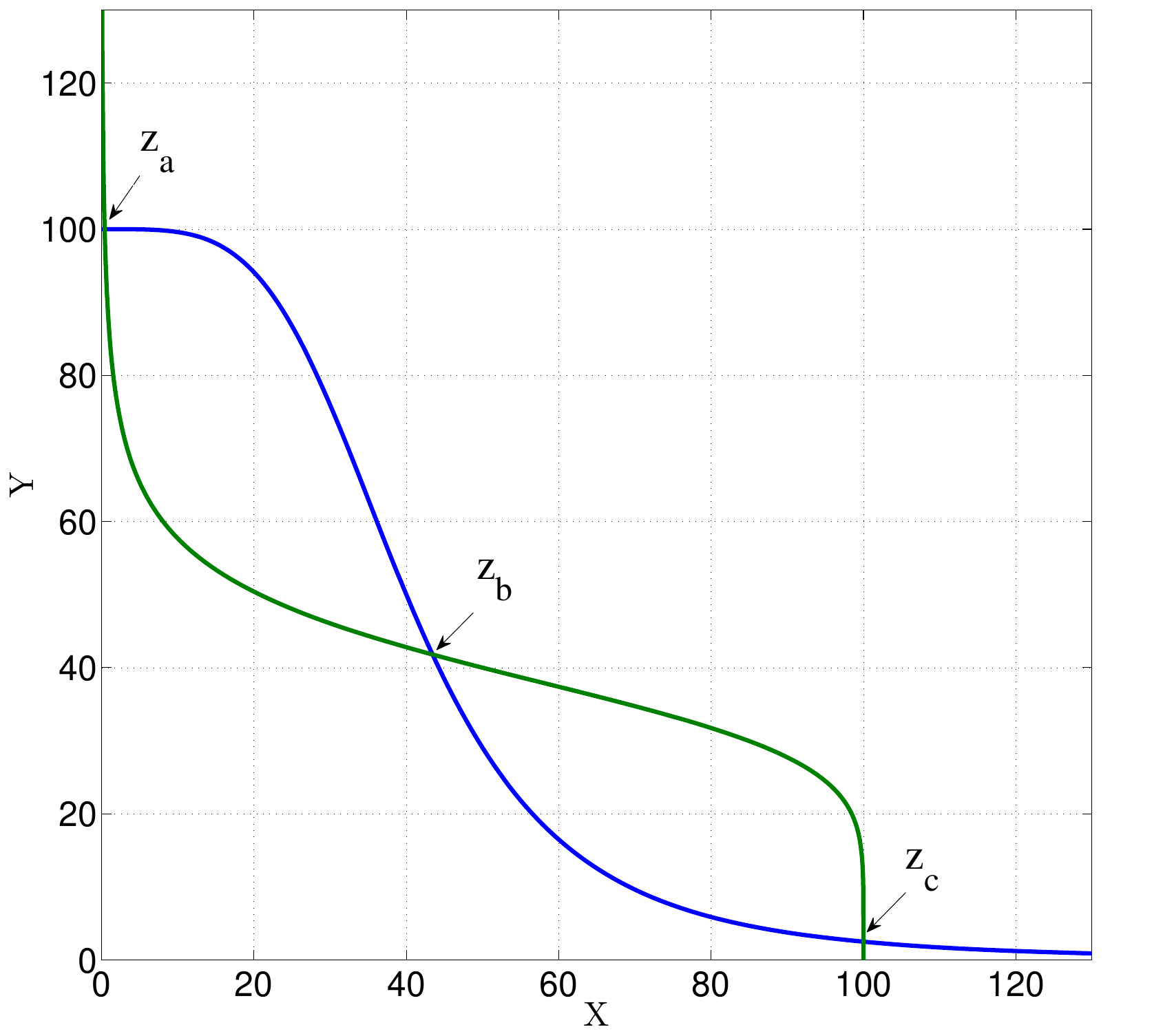}} 
		\goodgap \goodgap \goodgap \goodgap \goodgap \goodgap \goodgap \goodgap
		\subfigure[The set $A$ is an avoidance region contained in the region of attraction of the stable equilibria $x_a$ and $x_c$, $B$ is the target set around the unstable equilibrium $x_b$, and $C$ is the maintenance margin.]{\label{fig:sets}\includegraphics[scale = 0.213]{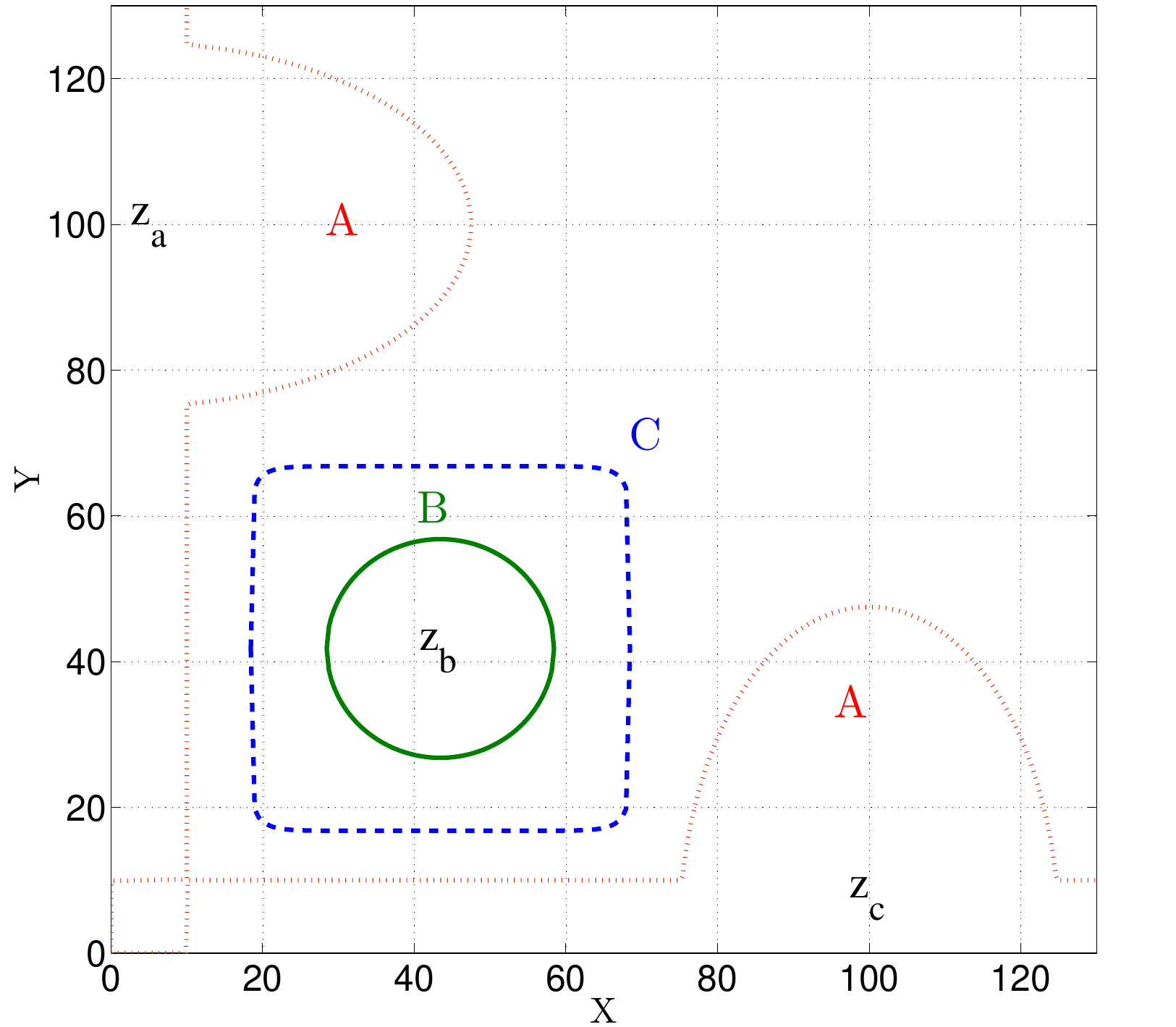}}
		\caption{State space of the biological switch \eqref{sde switch} with desired motion planning sets.}
		\label{fig:null_set}
	\end{figure}
	
	Driving bistable systems away from the commonly observed states may provide useful insights into their functional organization. Motivated by this, in this example we consider a motion planning task comprising two phases. First, we aim to steer the number of proteins to be at a target set around the unstable equilibrium at a certain time horizon, say $T_1$, by synthesizing appropriate input signals $\control{u}_x$ and $\control{u}_y$. During this phase we opt to avoid the region of attraction of the stable equilibria as well as low numbers for each protein; the latter justifies our Langevin model being well-posed in the region of interest. These target and avoid sets are denoted, respectively, by the closed sets $B$ and $A$ in Figure \ref{fig:sets}. In the second phase of the task, it is required to keep the molecular populations within a larger margin around the unstable equilibrium till sometime after $T_1$, say $T_2 > T_1$; Figure \ref{fig:sets} depicts this maintenance margin by the open set $C$. In the context of reachability, the second phase is known as \emph{viability} \cite{ref:AubinPrato-1998}. 

	In view of motion planning events introduced in Definition~\ref{def:event}, in particular \eqref{reach-event}, the above motion planning task can be expressed by $(A^c \Reach{T_1} B) \com (C \Reach{T_2} C)$. Let \eqsmall{\traj{Z}{t,z}{\control{u}}{\cdot} \Let \big[\traj{X}{t,x}{\control{u}}{\cdot},\traj{Y}{t,y}{\control{u}}{\cdot}\big]} be the joint process with initial condition $z \Let [x,y]$ and controller $\control{u} \Let [\control{u}_x, \control{u}_y]$. Following the framework of Section \ref{sec:DPP}, we consider the following value functions:
	\begin{subequations}
	\label{Vs sim}
	\begin{align}
		\label{V1 sim}
		&V_1(t,z) \Let \sup_{\control{u} \in \controlmaps{U}_t}\EE\Big[ \ind{B}\big( \traj{Z}{t,z}{\control{u}}{\tau^1_1}\big) \ind{C}\big( \traj{Z}{t,z}{\control{u}}{\tau^1_2}\big) \Big], \\ 
		\label{V2 sim}
		&V_2(t,z) \Let \sup_{\control{u} \in \controlmaps{U}_t}\EE\Big[ \ind{C}\big( \traj{Z}{t,z}{\control{u}}{\tau^2_2}\big) \Big], 
	\end{align}
	\end{subequations}	
	where the stopping times {$\tau^1_1, \tau^1_2,\tau^2_2$} are defined as in \eqref{Vk} with sets {\small $A_1 \Let A^c$} and {\small $A_2 \Let C$}. Let us recall that by Theorem~\ref{thm:SOC} the desired set of initial conditions in Definition~\ref{def:initial set} is indeed the $p$ superlevel set of $V_1$. The solution of the motion planning objective is the value function $V_1$ in \eqref{V1 sim}, which in view of Theorem \ref{thm:DPE} is characterized by the Dynkin PDE operator in the interior of $[0,T_1[ \times A^c$. However, due to the boundary condition \eqref{boundary pointwise}, we first need to compute $V_2$ in \eqref{V2 sim} to provide boundary conditions for $V_1$ according to
	\begin{align}
	\label{V1 sim boundary}
		\begin{cases}
			V_1(t,z) = V_2(t,z)\ind{B}(z),  \\
			\forall (t,z) \in [0,T_1]\times A \bigcup \{T_1\} \times \R^2.
		\end{cases}
	\end{align}
	Observe that the boundary condition for the value function $V_2$ is
	\begin{align*}
		V_2(t,z) = \ind{C}(z), \qquad \forall (t,z) \in [0, T_2] \times C^c \bigcup \{T_2\} \times \R^2.
	\end{align*}
	Therefore, we need to solve the PDE of $V_2$ backward from the terminal time $T_2$ to $T_1$ together with the above boundary condition. Then, the value function $V_1$ can be computed via solving the same PDE from $T_1$ to $0$ with the boundary condition \eqref{V1 sim boundary}. The Dynkin operator $\mathcal{L}^u$ reduces to	
	\begin{align*}
		\sup_{u \in \set{U}} \mathcal{L}^u \phi(t &,x,y) = \max_{u \in \set{U}} \Big[ \partial_t \phi  + \partial_x \phi\big( f(y,u_x) - \mu_x x \big) + \partial_y \phi \big( g(x,u_y) - \mu_y y \big) \\
		& \qquad \qquad \qquad + \frac{1}{2}\partial_x^2 \phi \big(f(y,u_x) + \mu_x x \big) + \frac{1}{2}\partial_y^2 \phi \big(g(x,u_y) + \mu_y y \big) \Big]\\
		& = \partial_t \phi - \big(\partial_x \phi-\frac{1}{2}\partial_x^2 \phi \big)\mu_x x - \big( \partial_y \phi - \frac{1}{2}\partial_y^2 \phi \big)\mu_y y \\
		& \qquad \qquad 
		+ \max_{u_x \in [\ul{u}_x, \ol{u}_x]} \big[ f(y,u_x) \big(\partial_x \phi + \frac{1}{2}\partial_x^2\phi \big) \big] + \max_{u_y \in [\ul{u}_y, \ol{u}_y]} \big[ g(x,u_y) \big(\partial_y \phi + \frac{1}{2}\partial_y^2\phi \big) \big].
	\end{align*}
	Thanks to the linearity of the drift term in $\control u$, an optimal control can be expressed in terms of derivatives of the value functions $V_1$ and $V_2$ as
	\begin{align*}
		\control{u}_x^*(t,x,y) = 
		\begin{cases}
			\ol{u}_{x}(t,x,y) \quad \text{if} ~ \partial_x V_i(t,x,y) + \frac{1}{2}\partial_x^2 V_i(t,x,y)\ge 0, \\
			\ul{u}_{x}(t,x,y) \quad \text{if} ~ \partial_x V_i(t,x,y) + \frac{1}{2}\partial_x^2 V_i(t,x,y) < 0,
		\end{cases} \\
		\control{u}_y^*(t,x,y) = 
		\begin{cases}
			\ol{u}_{y}(t,x,y) \quad \text{if} ~ \partial_y V_i(t,x,y) + \frac{1}{2}\partial_y^2 V_i(t,x,y)\ge 0, \\
			\ul{u}_{y}(t,x,y) \quad \text{if} ~ \partial_y V_i(t,x,y) + \frac{1}{2}\partial_y^2 V_i(t,x,y) < 0,
		\end{cases}
	\end{align*}
	where $i \in \{1,2\}$ corresponds to the phase of the motion. 	
	
	\begin{figure}[t!]
	\centering
		\subfigure[$V_2$ in case of full controllability over both production rates.]{\label{fig:stay}\includegraphics[scale = 0.16]{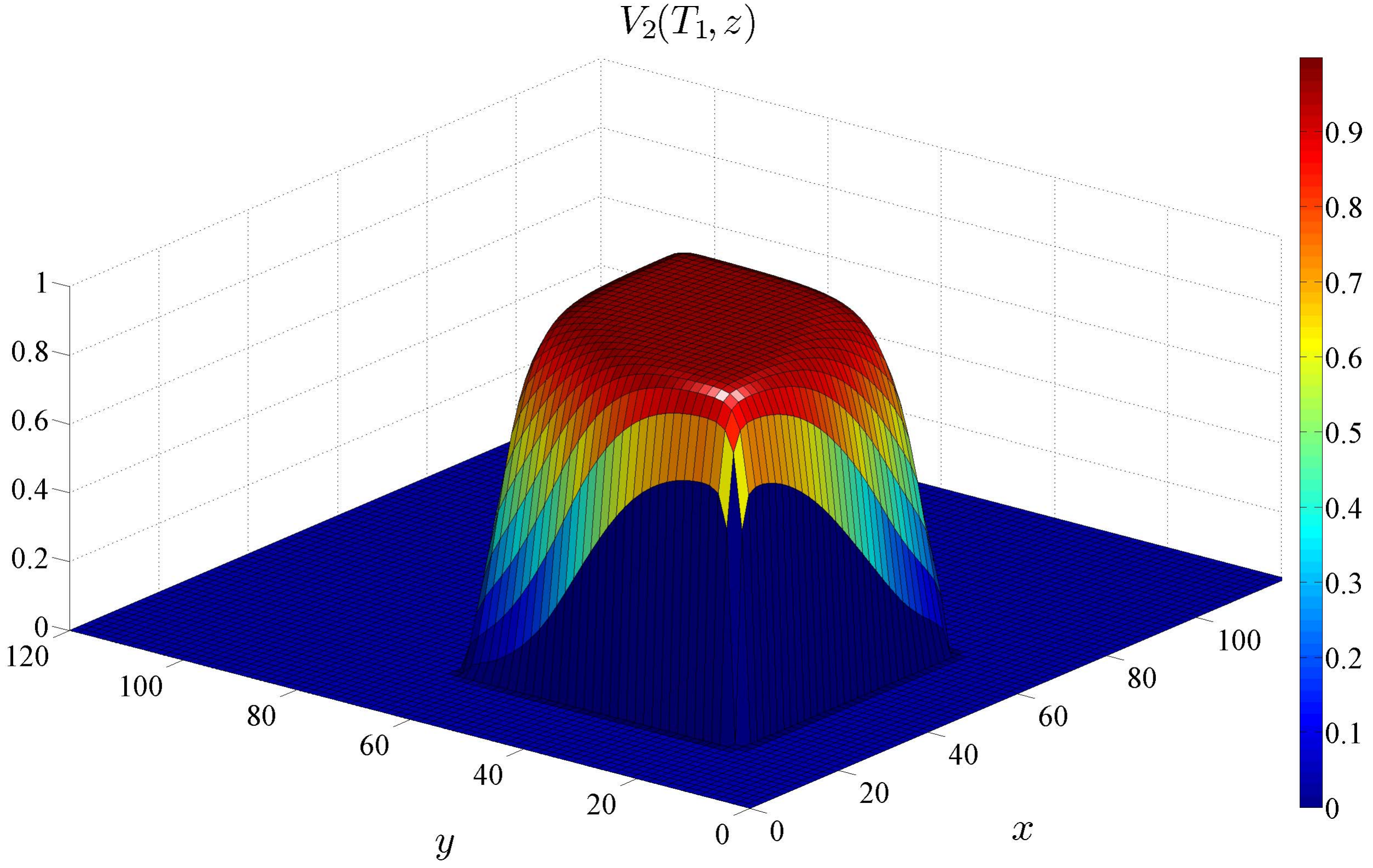}}
		\quad
		\subfigure[$V_2$ in case only the production rate of protein x is controllable.]{\label{fig:stay_y}\includegraphics[scale = 0.16]{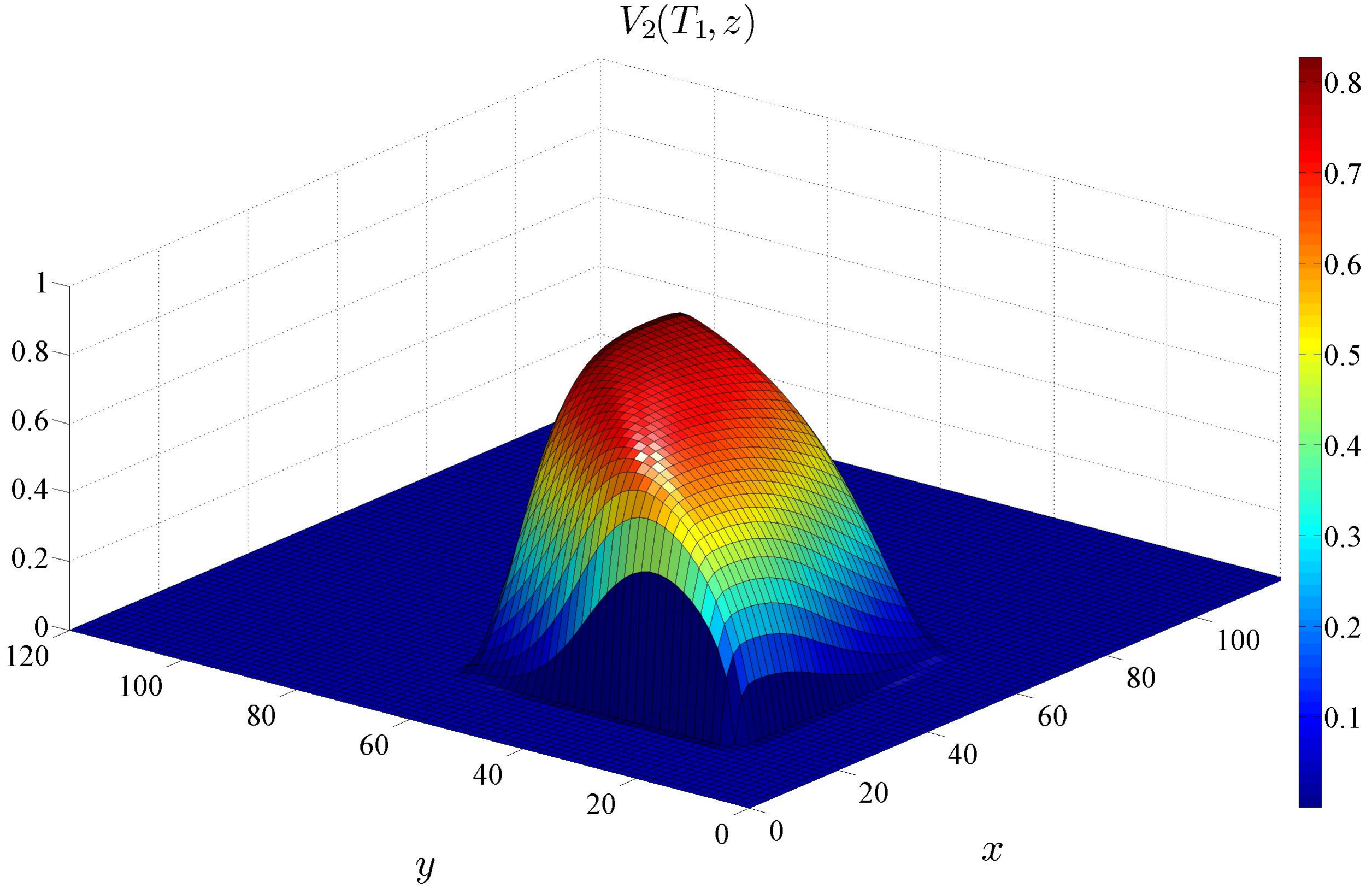}}
		\caption{The value function $V_2$ as defined in \eqref{V2 sim} corresponding to the probability of staying in $C$ for $60$ time units.}
		\label{fig:V2}
	\end{figure}
	
	\begin{figure}[t!]
		\centering
			\subfigure[$V_1$ in case of full controllability over the production rates.]{\label{fig:reach}\includegraphics[scale = 0.16]{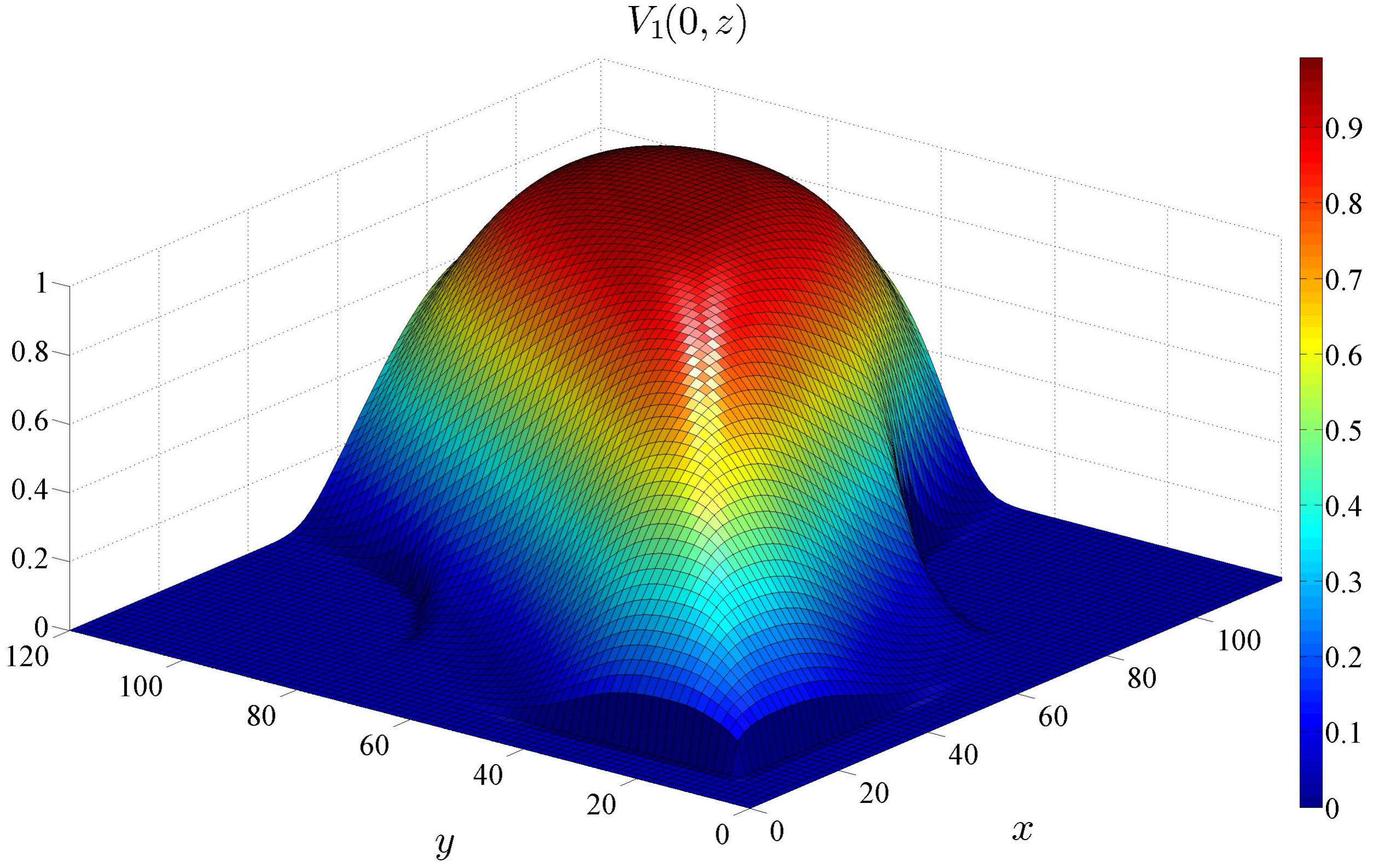}}
		\quad
			\subfigure[$V_1$ in case only the production rate of protein x is controllable.]{\label{fig:reach_y}\includegraphics[scale = 0.16]{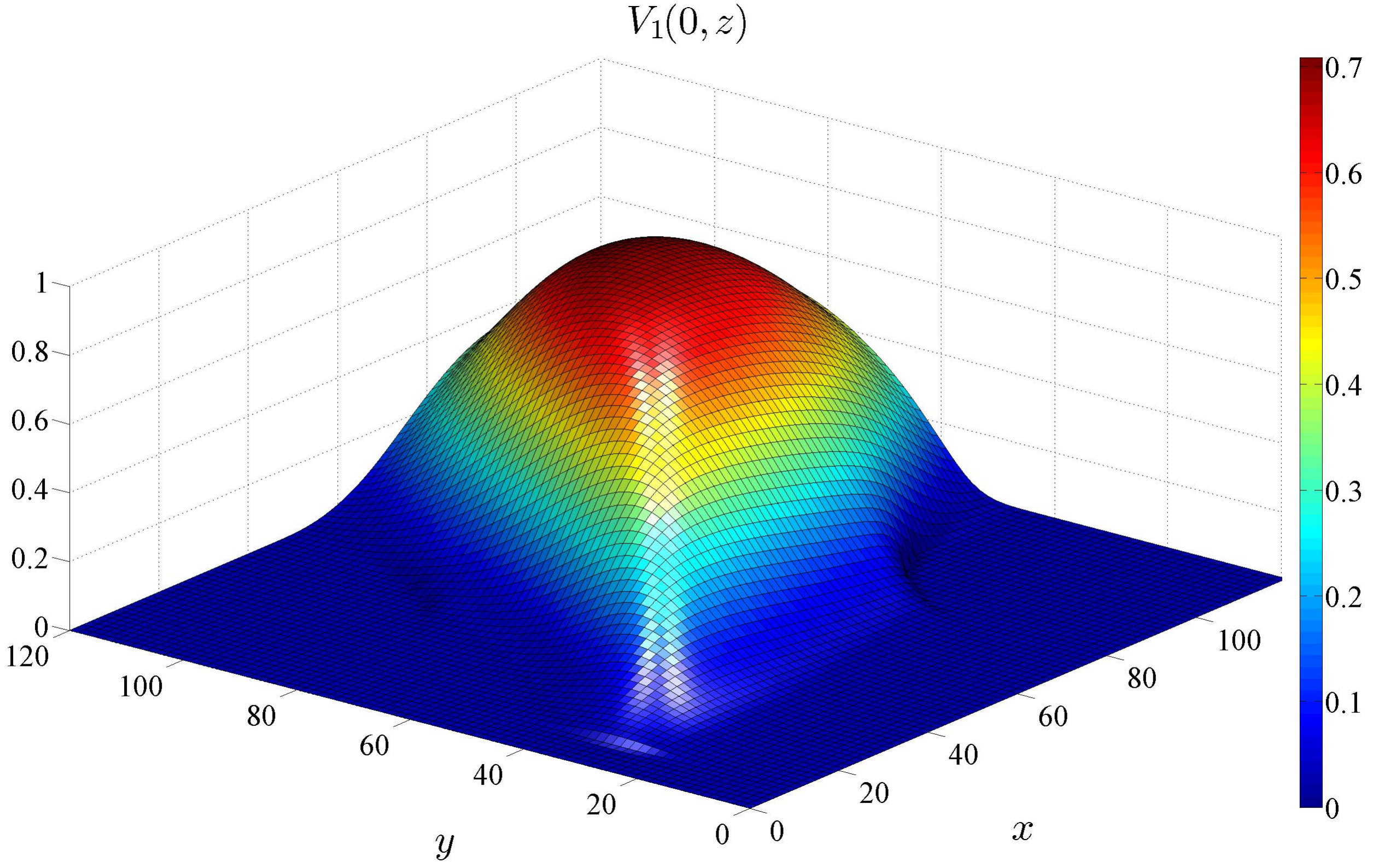}}
			\caption{The value function $V_1$ as defined in \eqref{V1 sim} corresponding to the probability of being at $B$ at $20$ time units and staying in $C$ afterward till $80$ time units, while avoiding $A$ throughout the motion.}
			\label{fig:V1}
	\end{figure}
			
	For this system we investigate two scenarios: One where full control over both production rates is possible, and one where only the production rate of protein $x$ can be controlled. Accordingly, in the first scenario we set $\ul{u}_x =  \ul{u}_y = 0$ and $\ol{u}_x = \ol{u}_y = 2$ while in the second we set $\ul{u_x}=0$, $\ol{u_x}=2$ and $\ol{u}_y = \ul{u}_y = 1$. Figure \ref{fig:V2} depicts the probability distribution of staying in set $C$ within the time horizon $[T_1, T_2] = [20, 80]$ for $60$ time units \footnote{Notice that the half-life of each protein is assumed to be 17.32 time units} in terms of the initial conditions $(x,y) \in \R^2$. Value function $V_2$ is zero outside set $C$, as the process has obviously left $C$ if it starts outside it. Figures \ref{fig:stay} and \ref{fig:stay_y} demonstrate the first and second scenarios, respectively. Note that in the second case the probability of success dramatically decreases in comparison to the first. This result indicates the importance of full controllability of the production rates for the achievement of the desired control objective.

	Figure \ref{fig:V1} depicts the probability of successively being at set $B$ at $T_1 = 60$ time units while avoiding $A$, and staying in set $C$ till $T_2 = 80$ time units thereafter. Since the objective is to avoid $A$ throughout the motion, the value function $V_1$ takes zero value on $A$. Figures \ref{fig:reach} and \ref{fig:reach_y} demonstrate the first and second control scenarios, respectively. It is easy to observe the non-smooth behavior of the value function $V_1$ on the boundary of set $A$ in Figure \ref{fig:reach_y}. 
	All simulations in this subsection were obtained using the Level Set Method Toolbox \cite{ref:Mitchell-toolbox} (version 1.1), with a grid $121 \times 121$ in the region of interest.

\section{Conclusion and Future Directions} \label{sec:conculusion}

	We introduced different notions of stochastic motion planning problems. Based on a class of stochastic optimal control problems, we characterized the set of initial conditions from which there exists an admissible control to execute the desired maneuver with probability no less than some pre-specified value. We then established a weak DPP in terms of auxiliary value functions. Subsequently, we focused on a case of diffusions as the solution of a controlled SDE, and investigated the required conditions to apply the proposed DPP. It turned out that invoking the DPP one can solve a series of PDEs in a recursive fashion to numerically approximate the desired initial set as well as the admissible control for the motion planning specifications. Finally, the performance of the proposed stochastic motion planning notions was illustrated for a biological switch network.

	For future work, as Theorem \ref{thm:DPP} holds for the broad class of stochastic processes whose sample paths are right continuous with left limits, we aim to study the required conditions of the proposed DPP (Assumptions \ref{a:DPP}) for a larger class of stochastic processes, e.g., controlled Markov jump-diffusions. Furthermore, motivated by the fact that full state measurements may not be available in practice, an interesting question is to address the motion planning objective with imperfect information, i.e., an admissible control would be only allowed to utilize the information of the process $Y_s \Let h(X_s)$ where $h:\R^d \ra \R^{d_y}$ is a given measurable mapping.

%
 \renewcommand{\thesection}{I}
 \section{Appendix: Technical Proofs of Sections \ref{sec:connection} \& \ref{sec:DPP}} \label{app-A}
 \setcounter{equation}{0}
 \numberwithin{equation}{section}

	We first start with a rather technical lemma showing that the sequential stopping times in Definition \ref{def:theta} are indeed well defined. 
	
	\begin{Lem}[Measurability]
	\label{fact:measurability}
		Consider a sequence of $(A_i)_{i=1}^n \subset \borel(\R^d)$ and $(t,x) \in \set{S}$. The sequential exit-time $\Theta^{A_{1:n}}_i(t,x)$ is an $\filtration_t$-stopping time for all $i \in \{1,\cdots,n\}$, i.e., $\big \{\Theta^{A_{1:n}}_i(t,x) \le s \big \} \in \sigalg_{t,s}$ for all $s \ge 0$. 
	\end{Lem}

	\begin{proof} 
	Let $\tau_A$ be the first exit-time from the set $A_i$:
	\begin{align}
	\label{exit-time}
	\tau_{A_i}(t,x) \Let \inf\{s \ge 0 ~:~ \traj{X}{t,x}{\control{u}}{t+s} \notin A_i\}.
	\end{align}
	We know that $\tau_A$ is an $\filtration_t$-stopping time \cite[Thm.\ 1.6, Chapter 2]{ref:EthKur-86}. Let $\omega(\cdot) \mapsto \shift_s \big(\omega(\cdot)\big) \Let \omega(s+\cdot)$ be the time-shift operator. From the definition it follows that for all $i \ge 0$
		\begin{align*}
			\Theta^{A_{1:n}}_{i+1} = \Theta^{A_{1:n}}_i + \tau_{A_i} \circ \shift_{\Theta^{A_{1:n}}_i}.
		\end{align*}
	Now the assertion follows directly in light of the measurability of the mapping $\shift$ and right continuity of the filtration $\filtration_t$; see \cite[Prop.\ 1.4, Chapter 2]{ref:EthKur-86} for more details in this regard.
	\end{proof}
	
	Before proceeding with the proof of Theorem \ref{thm:SOC}, we start with a fact which is an immediate consequence of right continuity of the process $\traj{X}{t,x}{\control{u}}{\cdot}$: 
	\begin{Fact}
	\label{fact:W}
		Fix a control $\control{u} \in \controlmaps{U}_t$ and an initial condition $(t,x) \in \set{S}$. Let $(A_i)_{i=1}^n\subset\borel(\R^d)$ be a sequence of open sets. Then, for all $i \in \{1,\cdots, n\}$ 
		\begin{align*}
			\traj{X}{t,x}{\control{u}}{\Theta^{A_{1:n}}_i} &\notin A_i, \qquad \text{on} \quad \big \{\Theta^{A_{1:n}}_i < \infty \big \},
		\end{align*}
		where $\big( \Theta^{A_{1:n}}_i \big)_{i=1}^n$ are the sequential exit-times in the sense of Definition \ref{def:theta}.
	\end{Fact}
	
\begin{proof}[Proof of Theorem \ref{thm:SOC}]
	We first show \eqref{prop V}. Observe that it suffices to prove	that
		\begin{align}
		\label{prop 1}
		\begin{array}{c}
			\Big\{\traj{X}{t,x}{\control{u}}{\cdot} \sat \big[ ( W_1 \Path G_1) \com \cdots \com (W_n \Path G_n) \big]_{\le T} \Big \} = \bigcap_{i = 1}^{n} \Big \{ \traj{X}{t,x}{\control{u}}{\eta_i} \in G_i \Big\}
		\end{array}
		\end{align}
	for all initial condition $(t,x)$ and control $\control{u}$, where the stopping time $\eta_i$ is as defined in \eqref{V}. Let $\omega$ belong to the left-hand side of \eqref{prop 1}. In view of the definition \eqref{path-event}, there exists a set of instants $(s_i)^{n}_{i=1} \subset [t,T]$ such that for all $i$, $\traj{X}{t,x}{\control{u}}{s_i}(\omega) \in G_i$ while $\traj{X}{t,x}{\control{u}}{r}(\omega) \in W_i \setminus G_i \teL B_i $ for all $r \in [s_{i-1}, s_i[$, where we set $s_0 = t$. It then follows by an induction argument that $\eta_i(\omega) = \Theta^{B_{1:n}}_i = s_i$, which immediately leads to $\traj{X}{t,x}{\control{u}}{\eta_i(\omega)}(\omega) \in G_i$ for all $i \le n$. This proves the relation $``\subset"$ between the left- and right-hand sides of \eqref{prop 1}. Now suppose that $\omega$ belongs to the right-hand side of \eqref{prop 1}. Then, we have $\traj{X}{t,x}{\control{u}}{\eta_i(\omega)}(\omega) \in G_i$ for all $i \le n$. In view of the definition of stopping times $\eta_i$ in \eqref{V}, it follows that $\traj{X}{t,x}{\control{u}}{r}(\omega) \in B_i \Let W_i \setminus G_i$ for all $r \in [\eta_{i-1}(\omega), \eta_i(\omega)[$. Introducing the time sequence $s_i \Let \eta_i(\omega)$ implies the relation $``\supset"$ between the left- and right-hand sides of \eqref{prop 1}. Together with preceding argument, this implies \eqref{prop 1}.

	To prove \eqref{prop Vtilde} we only need to show that
		\begin{align}
		\label{prop 2}
		\begin{array}{c}
			\Big \{ \traj{X}{t,x}{\control{u}}{\cdot} \sat (W_1 \Reach{T_1} G_1) \com \cdots \com (W_n\Reach{T_n} G_n)\Big\} 	= \bigcap_{i = 1}^{n} \Big \{ \traj{X}{t,x}{\control{u}}{\wt \eta_i} \in G_i \cap W_i \Big\}
		\end{array}
		\end{align}
	for all initial condition $(t,x)$ and controls $\control{u}$, where the stopping time $\wt \eta_i$ is introduced in \eqref{Vtilde}. To this end, let us fix $(t,x) \in \set{S}$ and $\control{u} \in \controlmaps{U}_t$, and assume that $\omega$ belongs to the left-hand side of \eqref{prop 2}. By definition \eqref{reach-event}, for all $i \le n$ we have $\traj{X}{t,x}{\control{u}}{T_i}(\omega) \in G_i$ and $\traj{X}{t,x}{\control{u}}{r}(\omega) \in W_i$ for all $r \in [T_{i-1}, T_{i}]$. By a straightforward induction, we see that $\wt \eta_i(\omega) = T_i$, and consequently $\traj{X}{t,x}{\control{u}}{\wt \eta_i(\omega)}(\omega) \in G_i \cap W_i$ for all $i \le n$. This establishes the relation $``\subset"$ between the left- and right-hand sides of \eqref{prop 2}. Now suppose $\omega$ belongs to the right-hand side of \eqref{prop 2}. Then, for all $i \le n$ we have $\traj{X}{t,x}{\control{u}}{\wt \eta_i(\omega)}(\omega) \in G_i \cap W_i$. By virtue of Fact \ref{fact:W} and an induction argument once again, it is guaranteed that $\wt \eta_i(\omega) = T_i$, and consequently it follows that $\traj{X}{t,x}{\control{u}}{T_i}(\omega) \in G_i$ and $\traj{X}{t,x}{\control{u}}{r}(\omega) \in W_i$ for all $r \in [T_{i-1}, T_{i}]$. This establishes the relation $``\supset"$ in \eqref{prop 2}, and the assertion follows.
	\end{proof}

	We now continue with the missing proof of Section \ref{sec:DPP}. Before proceeding with the proof of Theorem \ref{thm:DPP}, we need a preparatory lemma.
		
		\begin{Lem}
		\label{lem:lsc}
			Under Assumptions \ref{a:DPP}.\ref{a:process}\ref{a:process:exit-time} and \ref{a:DPP}.\ref{a:payoff}, the function $\set{S} \ni (t,x) \mapsto J_k(t,x;\control{u}) \in \R$ is lower semicontinuous for all $k \in \{1, \cdots, n\}$ and control $\control{u} \in \controlmaps{U}_0$.
		\end{Lem}
	
	\begin{proof} 
		Fix \(k\in\{1,\ldots, n\}\). It is obvious that the function \(J_k\) is uniformly bounded since $\ell_k$ are. Therefore, 
		\begin{align}
			\liminf_{(s,y) \ra (t,x)}  & J_k\big(s,y;\control{u}\big) = \liminf_{(s,y) \ra (t,x)} \EE \Big[ \prod_{i = k}^n \ell_i \big( \traj{X}{s,y}{\control{u}}{\tau_i^k(s,y)} \big) \Big] \notag \\
			& \ge \EE \Big[  \liminf_{(s,y) \ra (t,x)}\prod_{i = k}^n \ell_i \big( \traj{X}{s,y}{\control{u}}{\tau_i^k(s,y)} \big) \Big] \label{fatou} \\
			& \ge \EE \Big[  \prod_{i = k}^n \liminf_{(s,y) \ra (t,x)}\ell_i \big( \traj{X}{s,y}{\control{u}}{\tau_i^k(s,y)} \big) \Big] \notag \\
			& \ge \EE \Big[  \prod_{i = k}^n \ell_i \big( \traj{X}{t,x}{\control{u}}{\tau_i^k(s,y)} \big) \Big] = J_k(t,x;\control{u}), \label{eq}
		\end{align}
		where the inequality in \eqref{fatou} follows from the Fatou's lemma, and \eqref{eq} is a direct consequence of Assumptions \ref{a:DPP}.\ref{a:process}\ref{a:process:exit-time} and \ref{a:DPP}.\ref{a:payoff}
	\end{proof}

	\begin{proof}[Proof of Theorem \ref{thm:DPP}]
		The proof extends the main result of our earlier work \cite[Thm.\ 4.7]{ref:MohChatLyg-15} on the so-called reach-avoid maneuver. Let $u \in \controlmaps{U}_t$, $\theta \Let \theta^{\control{u}} \in \setofst{t,T}$, and $\control{u}_\theta$ be the random control as introduced in Assumption \ref{a:DPP}.\ref{a:process}\ref{a:process:markov} Then we have
	\begin{subequations}
	\label{tower}
		\begin{align}
			\EE &\bigg[ \prod_{i = k}^n  \ell_i \big( \traj{X}{t,x}{\control{u}}{\tau_i^k} \big) ~\Big|~ \sigalg_\theta \bigg] \\
			& =  \sum_{j = k}^{n+1} \ind{\{\tau_{j-1}^k \le \theta < \tau_j^k\}} \EE \bigg[ \prod_{i = j}^n \ell_i \big( \traj{X}{t,x}{\control{u}}{\tau_i^k} \big) ~\Big|~ \sigalg_\theta \bigg] \prod_{i = k}^{j-1}\ell_i\big( \traj{X}{t,x}{\control{u}}{\tau_i^k} \big) \nonumber\\
			\label{tower_1} & = \sum_{j = k}^{n+1} \ind{\{\tau_{j-1}^k \le \theta < \tau_j^k\}} J_j\big(\theta, \traj{X}{t,x}{\control{u}}{\theta}; \control{u}_\theta \big) \prod_{i = k}^{j-1}\ell_i\big( \traj{X}{t,x}{\control{u}}{\tau_i^k} \big) \\
			\label{tower_2} & \le \sum_{j = k}^{n+1} \ind{\{\tau_{j-1}^k \le \theta < \tau_j^k\}} V_j\big(\theta, \traj{X}{t,x}{\control{u}}{\theta}\big) \prod_{i = k}^{j-1}\ell_i\big( \traj{X}{t,x}{\control{u}}{\tau_i^k} \big)
		\end{align}
	\end{subequations}
	where \eqref{tower_1} follows from Assumption \ref{a:DPP}.\ref{a:process}\ref{a:process:markov} and right continuity of the process, and \eqref{tower_2} is due to the fact that $\control{u}_\theta \in \controlmaps{U}_{\theta(\omega)}$ for each realization $\omega \in \Omega$. In light of the tower property of conditional expectation \cite[Thm.\ 5.1]{ref:Kallenberg-97}, arbitrariness of $\control{u} \in \controlmaps{U}_t$, and obvious inequality $V_j \le {V^*_j}$, we arrive at \eqref{DPP-sup}.

	To prove \eqref{DPP-sub}, consider uniformly bounded upper semicontinuous functions $(\phi_j)_{j=k}^n$ such that $\phi_j \le {V_j}_*$ on $\set{S}$. Mimicking the ideas in the proof of our earlier work \cite[Thm.\ 4.7]{ref:MohChatLyg-15} and due to Lemma \ref{lem:lsc}, one can construct an admissible control $\control{u}^\eps_j$ for any $\eps > 0$ and $j \in \{k,\cdots,n\}$ such that
	\begin{align}
	\label{u_eps}
		\phi_j(t,x) - 3\eps \le J_j(t,x;\control{u}_j^\eps) \qquad \forall (t,x) \in \set{S}.
	\end{align}
	Let us fix $\control{u} \in \controlmaps{U}_t$ and $\eps > 0$, and define
	\begin{align}
	\label{v_eps}
		\control{v}^\eps \Let \ind{[t,\theta]}\control{u} + \ind{]\theta, T]} \sum_{j = k}^{n}  \ind{\{\tau^k_{j-1} \le \theta < \tau^k_j \} } \control{u}_j^\eps,
	\end{align}
	where $\control{u}^\eps_j$ satisfies \eqref{u_eps}. Notice that Assumption \ref{a:DPP}.\ref{a:control}\ ensures $\control{v}^\eps \in \controlmaps{U}_t$. By virtue of the tower property, Assumptions \ref{a:DPP}.\ref{a:process}\ref{a:process:causal} and \ref{a:DPP}.\ref{a:process}\ref{a:process:markov}, and the assertions in \eqref{u_eps} and \eqref{v_eps}, it follows
	\begin{align*}
		V_k&(t,x)   \geq J_k(t,x;\control{v}^\epsilon)  = \EE\bigg[ \EE \Big[ \prod_{i = k}^n \ell_i \big( \traj{X}{t,x}{\control{v}^\eps}{\tau_i^k} \big) ~\Big|~ \sigalg_\theta \Big] \bigg] \\   
		& = \EE \bigg[ \sum_{j = k}^{n+1} \ind{\{\tau_{j-1}^k \le \theta < \tau_j^k\}} J_j\big(\theta, \traj{X}{t,x}{\control{u}}{\theta}; \control{u}^\eps_j\big) \prod_{i = k}^{j-1}\ell_i\big( \traj{X}{t,x}{\control{u}}{\tau_i^k} \big) \bigg] \\
		& = \EE \bigg[ \sum_{j = k}^{n+1} \ind{\{\tau_{j-1}^k \le \theta < \tau_j^k\}} \Big(\phi_j\big(\theta, \traj{X}{t,x}{\control{u}}{\theta}\big) -3\eps \Big) \prod_{i = k}^{j-1}\ell_i\big( \traj{X}{t,x}{\control{u}}{\tau_i^k} \big) \bigg].
	\end{align*}
	Now, consider a sequence of increasing continuous functions $(\phi_j^m)_{m\in \N}$ that converges point-wise to ${V_j}_*$. The existence of such sequence is ensured by Lemma \ref{lem:lsc}, see \cite[Lemma 3.5]{Reny_PointwiseConvergence}. By boundedness of $(\ell_j)_{i=1}^n$ and the dominated convergence Theorem, we get
	\begin{align*}
		V_k( & t,x)  \geq \\
		&\EE \bigg[ \sum_{j = k}^{n+1} \ind{\{\tau_{j-1}^k \le \theta < \tau_j^k\}} \Big({V_j}_*\big(\theta, \traj{X}{t,x}{\control{u}}{\theta}\big) -3\eps \Big) \prod_{i = k}^{j-1}\ell_i\big( \traj{X}{t,x}{\control{u}}{\tau_i^k} \big) \bigg]
	\end{align*}
	Since $\control{u} \in \controlmaps{U}_t$ and $\eps >0$ are arbitrary, this leads to \eqref{DPP-sub}.
	\end{proof}

\renewcommand{\thesection}{II}
 \section{Appendix: Technical Proofs of Section \ref{sec:application}} \label{app-B}
 \setcounter{equation}{0}
 \numberwithin{equation}{section}

	
\begin{proof}[Proof of Proposition \ref{prop:exit-time continuity}]
	The key step in the proof relies on the two Assumptions \ref{a:SDE}.\ref{a:SDE:nondegenerate} and \ref{a:SDE:set}. There is a classical result on non-degenerate diffusion processes indicating that if the process starts from the tip of a cone, then it enters the cone with probability one \cite[Corollary 3.2, p.\ 65]{ref:Bass-1998}. This hints at the possibility that the aforementioned Assumptions together with almost sure continuity of the strong solution of the SDE \eqref{SDE} result in the continuity of sequential exit-times $\Theta^{A_{1:n}}_i$ and consequently $\tau_i$. In the following we shall formally work around this idea.

	Let us assume that $t_m \le t$ for notational simplicity, but one can effectively follow similar arguments for $t_m > t$. By the definition of the SDE \eqref{SDE},
		\begin{align*}
			\traj{X}{t_m,x_m}{\control{u}}{r} = \traj{X}{t_m,x_m}{\control{u}}{t} &+ \int_{t}^r f\big( \traj{X}{t_m,x_m}{\control{u}}{s}, u_s \big) \diff s \\
			&+ \int_{t}^r \sigma \big(\traj{X}{t_m,x_m}{\control{u}}{s}, u_s\big) \diff W_s, \quad \PP \text{-a.s.}
		\end{align*}
	By virtue of \cite[Thm.\ 2.5.9, p.\ 83]{Krylov_ControlledDiffusionProcesses}, for all $q \ge 1$ we have 
		\begin{align*}
			\EE & \Big[\sup_{r \in [t,T]} \big \| \traj{X}{t,x}{\control{u}}{r}  - \traj{X}{t_m,x_m}{\control{u}}{r}   \big \|^{2q} \Big] \\
			& \le C_1(q,T,K) \EE \Big[ \big \| x - \traj{X}{t_m,x_m}{\control{u}}{t}   \big \|^{2q} \Big]\\
			& \le 2^{2q-1}C_1(q,T,K) \EE \Big[ \|x-x_m\|^{2q} + \big \| x_m - \traj{X}{t_m,x_m}{\control{u}}{t}   \big \|^{2q} \Big],
		\end{align*}
	whence, in light of \cite[Corollary 2.5.12, p.\ 86]{Krylov_ControlledDiffusionProcesses}, we get
		\begin{align}
		\label{ineq-init}
			\EE \Big[\sup_{r \in [t,T]}& \big \| \traj{X}{t,x}{\control{u}}{r}  - \traj{X}{t_n,x_n}{\control{u}}{r}   \big \|^{2q} \Big] \\ \notag
			& \le C_2(q,T,K,\|x\|)\big(\|x-x_n\|^{2q} + |t-t_n|^q\big).
		\end{align}
	In the above inequalities, $K$ is the Lipschitz constant of $f$ and $\sigma$ mentioned in Assumption \ref{a:SDE}.\ref{a:SDE:lip}; $C_1$ and $C_2$ are constant depending on the indicated parameters. Hence, in view of Kolmogorov's continuity criterion \cite[Corollary 1 Chap.\ IV, p.\ 220]{ref:Protter-2005}, one may consider a version of the stochastic process $\traj{X}{t,x}{\control{u}}{\cdot}$ which is continuous in $(t,x)$ in the topology of uniform convergence on compacts. This leads to the fact that $\PP$-a.s, for any $\eps > 0 $, for all sufficiently large $m$,
		\begin{equation}
		\label{tube 1}
			\traj{X}{t_m,x_m}{\control{u}}{r} \in \ball{B}_\eps\big(\traj{X}{t_0,x_0}{\control{u}}{r}\big), \qquad \forall r \in [t_m,T],
		\end{equation}
	where $\ball{B}_\eps(y)$ denotes the ball centered at $y$ and radius $\eps$. For simplicity, let us define the shorthand $\tau_i^m \Let \tau_i(t_m,x_m)$.\footnote{This notation is only employed in this proof.} By the definition of $\tau_i$ and Definition \ref{def:theta}, since the set $A_i$ is open, we conclude that 
		\begin{equation}
		\label{tube 2}
			\exists \eps>0, \quad \bigcup_{s \in [\tau_{i-1}^0, \tau_i^0[} \ball{B}_{\eps}(\traj{X}{t_0,x_0}{\control{u}}{s}) \cap A_i^c = \emptyset \qquad \PP \text{-a.s.}
		\end{equation}	
	By definition $\tau_0^0 \Let \tau_0(t_0,x_0) = t_0 $. As an induction hypothesis, let us assume $\tau_{i-1}^0$ is $\PP$-a.s.\ continuous, and we proceed with the induction step. One can deduce that \eqref{tube 2} together with \eqref{tube 1} implies that $\PP$-a.s.\ for all sufficiently large $m$,
		\begin{equation*}
			\traj{X}{t_m,x_m}{\control{u}}{r} \in A_i, \qquad \forall r \in [t_m, \tau_i^0[.
		\end{equation*}
	In conjunction with $\PP$-a.s.\ continuity of sample paths, this immediately leads to
		\begin{align}
		\label{liminf tau_n}
			\liminf_{m \ra \infty} \tau_i^m & \Let \liminf_{m \ra \infty} \tau_i(t_m,x_m) \\
			& \ge \tau_i(t_0,x_0) \qquad \PP \text{-a.s.} \notag
		\end{align}
	On the other hand, as mentioned earlier, the Assumptions \ref{a:SDE}.\ref{a:SDE:nondegenerate} and \ref{a:SDE:set} imply that the set of sample paths that hit the boundary of $A_i$ and do not enter the set is negligible \cite[Corollary 3.2, p.\ 65]{ref:Bass-1998}. Thus, with probability one we have
		\begin{align*}
			\forall \delta > 0, \quad &\exists s \in \big[ \Theta^{A_{1:n}}_{i}(t_0,x_0), ~ \Theta^{A_{1:n}}_{i}(t_0,x_0) + \delta \big[ ~:~ \\
			& \traj{X}{t_0,x_0}{\control{u}}{s} \in A_i 
		\end{align*}
	Hence, in light of \eqref{tube 1}, $\PP$-a.s.\ there exists $\eps >0 $, possibly depending on $\delta$, such that for all sufficiently large $m$ we have 
		\begin{align*}
			\traj{X}{t_m,x_m}{\control{u}}{s} \in \ball{B}_{\eps}(\traj{X}{t_0,x_0}{\control{u}}{s}) 	\subset A_i^c
		\end{align*}
	Recalling the induction hypothesis, we note that in accordance with the definition of sequential stopping times $\Theta^{A_{1:n}}_{i}$, one can infer that $\Theta^{A_{1:n}}_{i}(t_m,x_m) \le s < \Theta^{A_{1:n}}_{i}(t_0,x_0) + \delta$. From arbitrariness of $\delta$ and the definition of $\tau_i$, this leads to
		\begin{align*} 
			\limsup_{m \ra \infty} \tau_i(t_m,x_m) & \Let \limsup_{m \ra \infty} \big(\Theta^{A_{1:n}}_{i}(t_m,x_m) \mn T_i \big)\\
			& \le \tau_i(t_0,x_0) \qquad \PP \text{-a.s.},
		\end{align*}
	where in conjunction with \eqref{liminf tau_n}, $\PP$-a.s.\ continuity of the map $(t,x) \mapsto \tau_i(t,x)$ at $(t_0,x_0)$ for any $i \in \{1,\cdots,n\}$ follows. The assertion follows by induction.
	
	The continuity of the mapping $(t,x) \mapsto \traj{X}{t,x}{\control{u}}{\tau_i(t,x)}$ follows immediately from the almost sure continuity of the stopping time $\tau_i(t,x)$ in conjunction with the almost sure continuity of the version of the stochastic process $\traj{X}{t,x}{\control{u}}{\cdot}$ in $(t,x)$; for the latter let us note again that Kolmogorov's continuity criterion guarantees the existence of such a version in light of \eqref{ineq-init}.	
\end{proof}
	
\begin{proof}[Proof of Theorem \ref{thm:DPE}]
	Here we briefly sketch the proof of the first assertion of the theorem (supersolution property), and refer the reader to \cite[Thm.\ 4.10]{ref:MohChatLyg-15} for details concerning the same technology to prove the theorem. 
	
	Note that any $\filtration_t$-progressively measurable $\control{u} \in \controlmaps{U}_t$ satisfies Assumptions~\ref{a:DPP}.\ref{a:control}. It is a classical result \cite[Chap.\ 7]{ref:Oksendal} that the strong solution $\traj{X}{t,x}{\control{u}}{\cdot}$ satisfies Assumptions \ref{a:DPP}.\ref{a:process}\ref{a:process:causal} and \ref{a:DPP}.\ref{a:process}\ref{a:process:markov} Furthermore, Proposition \ref{prop:exit-time continuity} together with almost sure path-continuity of the strong solution guarantees Assumption~\ref{a:DPP}.\ref{a:process}\ref{a:process:exit-time} Hence, having fulfilled all the required assumptions of Theorem \ref{thm:DPP}, we can employ the DPP \eqref{DPP}. For the sake of contraction to the first assertion, suppose there exists $(t_0,x_0) \in [0,T_k[ \times A_k$ such that $-\sup_{u\in\set U}\mathcal{L}^u {V_k}_*(t_0,x_0) < 0$ in the viscosity sense. That is, there exist a smooth function $\phi$ and $\delta > 0$ such that 
	\begin{align*}
	\begin{cases}
		\min\limits_{(t,x) \in \set S} \big({V_k}_* - \phi\big)(t,x) = \big({V_k}_*-\phi\big)(t_0,x_0) = 0,\\
		-\sup\limits_{u \in \set{U}} \mathcal{L}^u \phi(t_0,x_0) < -2\delta.
	\end{cases}
	\end{align*}
	Since $\phi$ is smooth, the map $(t,x) \mapsto \mathcal{L}^u \phi (t,x)$ is continuous for each $u \in \set U$. Therefore, there exist $u_0 \in \set{U}$ and $r>0$ such that 
	\begin{align*}
		-\mathcal{L}^{u_0} \phi(t,x) < -\delta, \quad \forall (t,x) \in \ball{B}_r(t_0,x_0).
	\end{align*}
	Let us define the stopping time $\theta(t,x)$ as the first exit time of trajectory $\traj{X}{t,x}{u_0}{\cdot}$ from the ball $\ball{B}_r(t_0,x_0)$. Note that by continuity of the solution process of the SDE \eqref{SDE}, it holds that $t < \theta(t,x)$ with probability 1 for all $(t,x) \in \ball{B}_r(t_0,x_0)$. Therefore, selecting $r>0$ sufficiently small so that $\ball{B}_r(t_0,x_0) \subset [0,T_k[ \times A_k$ and applying It\^{o}'s formula, we see that for all $(t,x) \in \ball{B}_r(t_0,x_0)$, we have 
	\begin{align}
	\label{phi-theta}
		\phi(t,x) < \EE\big[\phi\big(\theta(t,x),\traj{X}{t,x}{u}{\theta(t,x)}\big)\big].
	\end{align}
	Let us take a sequence $(t_m,x_m,V_k(t_m,x_m))_{m \in \set{N}}$ converging to $(t_0,x_0,{V_k}_*(t_0,x_0))$, i.e., $\phi(t_m,x_m) \ra \phi(t_0,x_0) = {V_k}_*(t_0,x_0).$ Due to the inequality \eqref{phi-theta}, for sufficiently large $m$ we have $$V(t_m,x_m) < \EE\big[{V_k}_*\big(\theta(t_m,x_m),\traj{X}{t_m,x_m}{u}{\theta(t_m,x_m)}\big)\big],$$ which, in view of the fact that $\theta(t_m,x_m)<\tau_k\mn T_k$, contradicts the DPP in \eqref{DPP-sup}. The \emph{subsolution} property is proved effectively in a similar fashion.
\end{proof}

	To provide boundary conditions, in particular in the viscosity sense \eqref{boundary visc}, we need some preliminaries as follows: 

	\begin{Fact}
	\label{fact:theta}
		Consider a control $\control{u} \in \controlmaps{U}_t$ and initial condition $(t,x) \in \set{S}$. Given a sequence of $(A_i)_{i=k}^n \subset \borel(\R^d)$ and stopping time $\theta \in \setofst{t,T}$, for all $k \in \{1,\cdots,n\}$ and $j \ge i \ge k$ we have
		\begin{align*}
			\Theta^{A_{k:n}}_{j}(t,x) = \Theta^{A_{i:n}}_{j}\big(\theta,\traj{X}{t,x}{\control{u}}{\theta}\big),  
		\end{align*}
		pointwise on the set $\big\{ \Theta^{A_{k:n}}_{i-1}(t,x) \le \theta < \Theta^{A_{k:n}}_{i}(t,x) \big\}$.
	\end{Fact}

	\begin{Lem}
	\label{lem:exit-time unif cont}
	Suppose that the conditions of Proposition \ref{prop:exit-time continuity} hold. Given a sequence of controls $(\control{u_m})_{m\in \N} \subset \controlmaps{U}$ and initial conditions $(t_m,x_m) \ra (t,x)$, we have
		\begin{align*}
			\lim_{m \ra \infty} \Big \| \traj{X}{t,x}{\control{u}_m}{\tau_i(t,x)} - \traj{X}{t_m,x_m}{\control{u}_m}{\tau_i(t_m,x_m)} \Big \| = 0 \quad \PP \text{-a.s.}, 
		\end{align*}
	where $\tau_i(t,x) \Let \Theta^{A_{1:n}}_i(t,x) \mn T_i$. 
	\end{Lem}
	Note that Lemma \ref{lem:exit-time unif cont} is indeed a stronger statement than Proposition \ref{prop:exit-time continuity} as the desired continuity is required uniformly with respective to the control. Let us highlight that the stopping times $\tau_i(t,x)$ and $\tau_i(t_m,x_m)$ are both effected by the control $\control{u}_m$. But nonetheless, the mapping $(t,x) \mapsto \traj{X}{t,x}{\control{u_m}}{\tau_i}$ is almost surely continuous irrespective of the controls $(\control{u}_m)_{m\in \N}$. For the proof we refer to an identical technique used in \cite[Lemma 4.11]{ref:MohChatLyg-15}

\begin{proof}[Proof of Proposition \ref{prop:boundary}]
	The boundary condition in \eqref{boundary pointwise} is an immediate consequence of the definition of the sequential exit-times introduced in Definition \ref{def:theta}. Namely, for any initial state $x \in A_k^c$ we have $\Theta^{A_{k:n}}_k(t,x) = t$, and in light of Fact \ref{fact:theta} for all $i \in \{k, \cdots, n\}$ 
		\begin{align*}
			\Theta^{A_{k:n}}_i(t,x) &= \Theta^{A_{k+1:n}}_i(t,x), \\ 
			\forall (t,x) ~ &\in [0,T_k] \times  A_k^c \bigcup \{T_k\} \times \R^d.
		\end{align*}
	Since $\tau_k^k = t$ for the above initial conditions, then $\traj{X}{t,x}{\control{u}}{\tau_k^k} = x$ which yields to \eqref{boundary pointwise}. 

	For the boundary conditions \eqref{boundary visc}, we show the first assertion; the second follows similarly. Let $\big(t_m,x_m\big) \ra \big(t,x\big)$ where $t_m < T_k$ and $x_m \in A_k$. Invoking the DPP in Theorem \ref{thm:DPP} and introducing $\theta \Let \tau_{k+1}^k$ in \eqref{DPP-sup}, we reach
		\begin{align*}
			V_k(t_m,x_m) & \le \sup_{\control{u} \in \controlmaps{U}_t} \EE \Big[ {V}^*_{k+1} \big(\tau_k^k, \traj{X}{t_m,x_m}{\control{u}}{\tau_k^k}\big) \ell_k\big( \traj{X}{t_m,x_m}{\control{u}}{\tau_k^k}\big) \Big].
		\end{align*}
	Note that one can replace a sequence of controls in the above inequalities to attain the supremum. This sequence, of course, depends on the initial condition $(t_m,x_m)$. Hence, let us denote it via two indices $(\control{u}_{m,j})_{j \in \N}$. One can deduce that there exists a subsequence of $(\control{u}_{m_j})_{j \in \N}$ such that 
		\begin{align}
			\lim_{m \ua \infty}& V_k(t_m,x_m)  \notag  \\
			&\le \lim_{m \ua \infty}\lim_{j \ua \infty} \EE \Big[ {V}^*_{k+1} \big(\tau_k^k, \traj{X}{t_m,x_m}{\control{u}_{m,j}}{\tau_k^k}\big) \ell_k\big( \traj{X}{t_m,x_m}{\control{u}_{m,j}}{\tau_k^k}\big) \Big] \nonumber\\
			& \le \lim_{j \ua \infty} \EE \Big[ {V}^*_{k+1} \big(\tau_k^k, \traj{X}{t_j,x_j}{\control{u}_{m_j}}{\tau_k^k}\big) \ell_k\big( \traj{X}{t_j,x_j}{\control{u}_{m_j}}{\tau_k^k}\big) \Big] \nonumber\\
			& \label{lem 1} \le  \EE \Big[\lim_{j \ua \infty} {V}^*_{k+1} \big(\tau_k^k, \traj{X}{t_j,x_j}{\control{u}_{m_j}}{\tau_k^k}\big) \ell^*_k\big( \traj{X}{t_j,x_j}{\control{u}_{m_j}}{\tau_k^k}\big) \Big]\\
			& \label{lem 2} = V^*_{k+1}(t,x) \ell^*_k(x)
		\end{align}
	where \eqref{lem 1} and \eqref{lem 2} follow, respectively, from Fatou's lemma and the uniform continuity assertion in Lemma \ref{lem:exit-time unif cont}. Let us recall that by Lemma \ref{lem:exit-time unif cont} we know $\tau_k^k(t_j,x_j)\ra \tau_k^k(t,x) = t$ as $j \ra \infty$ uniformly with respect to the controls $(\control{u}_{m_j})_{j \in \N}$. Similar analysis would follow for the second part of \eqref{boundary visc} by using the other side of DPP in \eqref{DPP-sub}.
\end{proof}
	
\bibliographystyle{amsalpha} 
\bibliography{ref,ref_MohajerinEsfahani}

 \end{document}